\documentclass{amsart}

\usepackage{graphicx}

\usepackage[cmtip,all]{xy}

\usepackage{amsmath}
\usepackage{mathrsfs}
\usepackage{txfonts}
\usepackage{amssymb}

\newtheorem{theorem}{Theorem}[section]
\newtheorem{lemma}[theorem]{Lemma}
\newtheorem{proposition}[theorem]{Proposition}
\newtheorem{corollary}[theorem]{Corollary}
\theoremstyle{definition}
\newtheorem{definition}[theorem]{Definition}

\theoremstyle{remark}
\newtheorem{remark}[theorem]{Remark}

\numberwithin{equation}{section}

\begin{document}

\title{Interpolation problems in subdiagonal algebras}

 \author{Guoxing Ji}
\address{School  of Mathematics and Statistics,   Shaanxi Normal University,   Xian, 710119, People's  Republic of  China}
\email{gxji@snnu.edu.cn}
\thanks{This research was supported by the National Natural    Science Foundation of China(No. 12271323)}






\begin{abstract}
Let $\mathfrak A$ be a   subdiagonal algebra  with diagonal $\mathfrak D$ in a $\sigma$-finite von Neumann algebra $\mathcal M$ with respect to a faithful normal conditional expectation $\Phi$. We mainly consider the interpolation problem in $\mathfrak A$ with the universal factorization property. We determine when  a finitely generated left ideal in $\mathfrak A$ is trivial. By constructing  a periodic flow on $\mathcal M$ according to a type 1 subdiagonal algebra, we show that type 1 subdiagonal algebras coincide with analytic operator algebras associated with periodic flows in von Neumann algebras. This enables us to present a form  decomposition of a type 1 subdiagonal algebra. As an application, we deduce a noncommutative operator-theoretic variant of the Corona theorem for type 1 subdiagonal algebras.
\end{abstract}

\maketitle


.

\section{Introduction}
  The  Corona theorem says that the open unit disc $\mathbb D=\{z\in\mathbb C:|z|<1\}$ is dense in the maximal ideal space of all bounded analytic functions $H^{\infty}(\mathbb T)$ on $\mathbb D$. This is equivalent to the well known  Carleson theorem(\cite{car}): If  functions $\{f_k:1\leq k\leq N\}\subseteq H^{\infty}(\mathbb T)$ satisfy
 \begin{equation}\label{f11}
 \sum\limits_{k=1}^N|f_k(z)|\geq \varepsilon, \ \ \ \  \forall z\in\mathbb D
  \end{equation} for a positive constant $\varepsilon$,  then there exist functions $\{g_k: 1\leq k\leq N\}\subseteq H^{\infty}(\mathbb T)$ such that
 $\sum_{k=1}^Ng_kf_k=1$. This leads to the following operator-theoretic variant. Let $T_f$ denote the  Toeplitz operator on $H^2(\mathbb T)$ for any function $f\in L^{\infty}(\mathbb T)$. If  functions $\{f_k:1\leq k\leq N\}\subseteq H^{\infty}(\mathbb T)$ satisfy
 \begin{equation}\label{f12}
 \sum\limits_{k=1}^N\|T_{f_k}^*x\|\geq \varepsilon\|x\|, \ \ \ \  \forall x\in H^2(\mathbb T)
  \end{equation} for a positive constant $\varepsilon$,  then there exist functions $\{g_k:1\leq k\leq N\}\subseteq H^{\infty}(\mathbb T)$ such that $\sum_{k=1}^Nf_kg_k=1$. This  firstly appeared in \cite{cob} by utilizing the Corona theorem.  Arveson gave a natural operator-theoretic proof of this result which he called an operator-theoretic variant of the Corona theorem in \cite{arv3}. This also leads to the following  general question in a unital Banach algebra $\mathfrak A$(cf.\cite{cob}). Provide necessary and sufficient conditions on the set $\{A_k:1\leq k\leq N\}\subseteq \mathfrak A$ for the existence of a set $\{B_k:1\leq k\leq N\}\subseteq \mathfrak A$  satisfying
  \begin{equation}\label{f13}
  \sum_{k=1}^NB_kA_k=I.
  \end{equation} The findings of solutions $\{B_k:1\leq k\leq N\}\subseteq \mathfrak A$ for (\ref{f13}) will be called the interpolation problem for $\{A_k:1\leq k\leq N\}.$   For further  detail, we refer the readers to \cite{cob}. In fact, if $\mathfrak A\subseteq \mathcal B(\mathcal H)$, where $\mathcal B(\mathcal H)$ is  the algebra of all bounded linear operators on a complex Hilbert space $\mathcal H$, then a necessary condition may be expressed in terms of the invariant  projections of $\mathfrak A$. Specifically, for any projection $P\in\mathcal B(\mathcal H)$ such that $(I-P)AP=0$,$ \forall A\in\mathfrak A$, if (\ref{f13}) holds, then $\sum_{k=1}^N\|(I-P)A_kx\|\geq\varepsilon\|(I-P)x\|$ for all $x\in \mathcal H$; or equivalently, $\sum_{k=1}^N\|(I-P)A_kx\|^2\geq\varepsilon^2\|(I-P)x\|^2$, for all $x\in \mathcal H$(cf.$(4.1)$ and $(4.2)$ in \cite{arv3}). In fact, Arveson proved  that for a nest algebra $\mbox{alg}\{P_n: -\infty\leq n\leq +\infty\}=\{A\in\mathcal B(\mathcal H): (I-P_n)AP_n=0,-\infty\leq n\leq +\infty\}$, where $P_n<P_{n+1}$ satisfying $ P_{-\infty}=0$ and $P_{+\infty}=I$, these conditions for all $P_n$ are sufficient(cf.\cite[Theorem 4.3]{arv3}). As an application, he deduced an  operator-theoretic variant of the Corona theorem(cf.\cite[Theorem 6.3]{arv3}).  We note that two key facts are used. One is the  universal factorization property for the nest algebra $\mbox{alg}\{P_n: -\infty\leq n\leq +\infty\}$. Another is the existence of a conditional expectation  from $\mathcal B(L^2(\mathbb T))$ onto the algebra $\{L_f:f\in L^{\infty}(\mathbb T)\}$.
  From the viewpoint of von Neumann algebras, the following explanation can be given. Both $H^{\infty}(\mathbb T) \subseteq L^{\infty}(\mathbb T)$ and $\mbox{alg}\{P_n: -\infty\leq n\leq +\infty\}\subseteq \mathcal B(\mathcal H)$ are subdiagonal algebras, which are  introduced in \cite{arv1} by Arveson,   with the universal factorization property.  Moreover, $\mathcal B(L^2(\mathbb T))$ is just the commutant
  of the diagonal $\mathbb C I$ of $\{L_f:f\in H^{\infty}(\mathbb T)\}$.  Motivated by this idea, we may consider the interpolation problem in  subdiagonal algebras   in  $\sigma$-finite von Neumann algebras  with the universal factorization property. Furthermore,   using Haagerup's noncommutative $L^p$-spaces, we construct  the left and right representations $L(\mathcal M)$ and $R(\mathcal M)$  of a  von Neumann algebra $\mathcal M$ on the noncommutative $L^2$ spaces $L^2(\mathcal M)$.   If $\mathfrak A$ is a type 1 subdiagonal algebra with diagonal $\mathfrak D$,    we then construct a conditional expectation from the commutant  of   $L(\mathfrak D)$(resp. $R(\mathfrak D)$    onto $R(\mathcal M)$(resp. $L(\mathcal M)$). In particular, we deduce a noncommutative operator-theoretic variant of the Corona  theorem for type 1 subdiagonal algebras. We recall some notions.

Let    $\mathcal M$ be   a $\sigma$-finite von Neumann algebra  on a complex Hilbert space $\mathcal H$ and $\mathfrak A\subseteq \mathcal M$ a $\sigma$-weakly closed unital subalgebra of $\mathcal M$.
Assume that  there exists
 a faithful normal conditional expectation $\Phi$  from $\mathcal M$ onto  the diagonal  $\mathfrak D=\mathfrak A\cap \mathfrak A^*$. If $\mathfrak A$ satisfies

(i)   $\Phi$ is  a homomorphism on $\mathfrak A$ and

(ii) $\mathfrak A$ + $\mathfrak A^{*}$ is $\sigma$-weakly dense in $\mathcal M$,\\
then we say that  $\mathfrak A$   is   a subdiagonal algebra   with respect to $\Phi$.

If there is not any subdiagonal algebra with respect to $\Phi$ containing $\mathfrak A$  in $\mathcal M$, then $\mathfrak A$ is said to be
  a maximal subdiagonal algebra  with respect to $\Phi$.
    Put $\mathfrak A_{0}$ = $\{X \in \mathfrak A : \Phi(X) =0\}$ and $\mathfrak A_{m} = \{X \in \mathcal M : \Phi(AXB) = \Phi(BXA) =0, \forall A \in \mathfrak A, \forall B \in \mathfrak A_{0}\}$.
By \cite[Theorem 2.2.1]{arv1},   $\mathfrak A_{m}$ is the unique maximal subdiagonal algebra of $\mathcal M$ with respect to $\Phi$ containing $\mathfrak A$. We say that $\mathfrak A$ is a subdiagonal algebra always mean that $\mathfrak A$ is a subdiagonal algebra with respect to a faithful normal  conditional expectation $\Phi$.

Since $\mathcal M$ is $\sigma$-finite, there exists a  faithful normal state $\varphi$  on $\mathcal M$ such that $\varphi\circ\Phi=\varphi$. It is well known that there exists the modular automorphism group  $\{\sigma_t^{\varphi}:t\in\mathbb R\}$  of $\mathcal M$ associated with $\varphi$ by Tomita-Takesaki theory. We note that a subdiagonal algebra $\mathfrak A$ is maximal subdiagonal if and only if it is $\{\sigma_t^{\varphi}:t\in\mathbb R\}$ invariant(\cite{jos1,xu}).  Haagerup  gave the noncommutative $L^p$ spaces associated with a $\sigma$-finite von Neumann algebra $\mathcal M$(cf. \cite{haa, ter}) based on this theory.
The crossed product $\mathcal N=\mathcal M\rtimes_{\sigma^{\varphi}} \mathbb{R}$ of $\mathcal M$ by $\mathbb {R}$ with respect to $\sigma^{\varphi}$ is a semifinite von Neumann algebra and there is the normal faithful semifinite trace $\tau$ on $\mathcal N$ satisfying $$\tau\circ \theta_s=e^{-s}\tau,  \forall s\in\mathbb {R},$$ where $\theta$  is the dual action of $\mathbb {R}$ on $\mathcal N$.

The Haagerup noncommutative $L^p$ space $L^p(\mathcal M)$ for each $0< p\leq \infty$ is defined as the set of all $\tau$-measurable  operators $x$   affiliated with $\mathcal N$ satisfying $$\theta_s(x)=e^{-\frac{s}{p}}x,  \forall  s\in\mathbb {R}.$$
Denote by $\mbox{tr}$ the  linear bijection  between the predual $\mathcal M_*$ of $\mathcal M$ and $L^1(\mathcal M)$.
 We know that $L^p(\mathcal M)$ is a Banach space with norm $\|x\|_p=(\mbox{tr}(|x|^p))^{\frac1p}$ for any $x\in L^p(\mathcal M)$ for $1\leq p<\infty$.   By  the inner product $\langle a,b\rangle=\mbox{tr}(b^*a)$, $ \forall a,b \in L^2(\mathcal M)$, $L^2(\mathcal M)$ becomes  a Hilbert space.
As in \cite{haa}, we may define left and right representations of $\mathcal M$ on $L^p(\mathcal M)$, that is, we define the operators $L_A$ and $R_A$ on $L^p(\mathcal M)$ ($1\leq p<\infty)$  by $L_Ax=Ax$ and $R_Ax=xA$ for all  $A\in\mathcal M$ and $x\in L^p(\mathcal M)$.
Then $A\to L_A$ (resp. $A\to R_A$) is a faithful representation (resp. anti-representation) of $\mathcal M$ on $L^2(\mathcal M)$.
For any subset $S\subseteq \mathcal M$, we put $L(S) =\{L_{A}:A\in S\}$ and $R(S) =\{R_{A}:A\in S\}$, respectively.
We may identify $\mathcal M$ with $L(\mathcal M)=\{L_A:A\in \mathcal M\}$(resp. $R(\mathcal M)=\{R_A:A\in \mathcal M\}$) on $L^2(\mathcal M)$ as needed.

Let $h_0$ be the noncommutative Radon-Nikodym derivative of the dual weight of $\varphi$ on $\mathcal N$ with respect to $\tau$.
For any $X\in\mathcal M$, we have $\varphi(X)=\mbox{tr}(h_0^{\theta}Xh_0^{1-\theta})$ and $\sigma_t^{\varphi}(X)=h_0^{it}Xh_0^{-it},\ \forall t\in \mathbb {R}, \forall \theta\in[0,1]$ (cf. \cite{kos1} \cite{kos2}). 
The noncommutative $H^p$  and    $ H_0^p$   in $L^p(\mathcal M)$ for any $1\leq p<\infty$  are defined to be  $H^p=H^p(\mathfrak A)=[h_0^{\frac{\theta}{p}}\mathfrak Ah_0^{\frac{1-\theta}p}]_p$ and $H_0^p=H_0^p(\mathfrak A)=[h_0^{\frac{\theta}{p}}\mathfrak A_0h_0^{\frac{1-\theta}p}]_p$, $\forall \theta\in[0,1]$ from \cite[Definition 2.6]{ji1} and \cite[Proposition 2.1]{ji2}.
If $p=\infty$, then we identify  $H^{\infty}$ as $\mathfrak A$ and identify $H_0^{\infty}$ as $\mathfrak A_0$. In general,   for a subset $S$ in a Banach space, we usually denote by $[S]$ the closed linear span.  We now  recall the notion of internal column(resp. row) sum give in \cite{junge}.
 Let $X$ be  a subspace of $L^p(\mathcal M)$ and  $\{X_n:n\geq1\}$  a family of subspaces of $X$ such that  $X=[X_n: n\geq1]$. If  $X_n^*X_m=0$(resp. $X_nX_m^*=0$) for any  $n\not=m$,
then we say that $X$ is the internal column(resp. row) $L^p$-sum and denoted by  $X=\oplus_{n\geq1}^{col}X_n$(resp. $X=\oplus_{n\geq1}^{row}X_n$).
If $p=\infty$, then we also assume that $X$ and $X_n$ are $\sigma$-weakly  closed  and   the closed linear span  is taken in  the $\sigma$-weak topology.

In this paper, we consider interpolation problems in a subdiagonal algebra $\mathfrak A$  with respect to a faithful normal conditional expectation $\Phi$ with the universal factorization  property.  In Section 2, we give   natural conditions for which interpolation problem $(1.3)$ for $\{A_k:1\leq k\leq N\}\subseteq \mathfrak A$ is solvable. In Section 3, we apply this result to a nest subalgebra  of $\mathcal M$ with a nest order isomorphic to a sublattice of the completely ordered lattice $\mathbb Z\cup\{-\infty,+\infty\}=\{n:-\infty\leq n\leq+\infty\}$. This generalizes  Arveson's result  on  nest algebras of $\mathcal B(\mathcal H)$ in \cite{arv3}. In particular, we get a noncommutative operator-theoretic variant of the Corona theorem for these algebras. To give a noncommutative operator-theoretical variant of the Corona theorem, we consider type 1 subdiagonal algebras and analytic operator algebras determined by periodic flows in Section 4.
           By constructing  a periodic flow on $\mathcal M$, we find that type 1 subdiagonal algebras coincide with analytic operator algebras associated with periodic flows. This enables us to present a form  decomposition for these subdiagonal algebras. Furthermore,  we construct a conditional expectation from the commutant $(L(\mathfrak D))^{\prime}$(resp. $(R(\mathfrak D))^{\prime}$) of $L(\mathfrak D)$(resp. $R(\mathfrak D)$) onto $R(\mathcal M)$(resp. $L(\mathcal M)$) for a type 1 subdiagonal algebra  with the  $(1,1)$-form, which may  be regarded as a noncommutative analogues of Arveson's expectation theorem(\cite[Proposition 5.1]{arv3}).   In Section 5, we give a noncommutative operator-theoretic variant of the Corona theorem for type 1 subdiagonal algebras.
 \section{Interpolation in subdiagonal algebras with the universal  factorization property}
  Let $\mathfrak A$ be a unital   subalgebra of  $\mathcal M$ and $S\in\mathcal M$  invertible. If there exists a unitary operator  $U\in\mathcal M$  such that $U^*S, S^{-1}U\in\mathfrak A$, then we say that $S$ has a    factorization in $\mathfrak A$.
   If every invertible operator in $\mathcal M$ has a   factorization in $\mathfrak A$, then we say that $\mathfrak A$  has the universal  factorization property. Let $\mathfrak A\subseteq \mathcal M$ be a subdiagonal  algebra.  If $\mathfrak A$ has the universal  factorization property, then $\mathfrak A$ is a maximal subdiagonal algebra by \cite[Proposition 2.1]{jd}. We consider the interpolation problem in a subdiagonal algebra $\mathfrak A$ with the universal factorization property.

Let $P_0$ and $P_1$ be  projections from $L^2(\mathcal M)$ onto $H^2$ and $H_0^2=[H^2\mathfrak A_0]$ respectively. It is known that an operator $A\in\mathcal M$ is in $\mathfrak A$ if and only if  $(I-P_0)AP_0=0$ or equivalently, $(I-P_1)AP_1=0$(cf. \cite[Theorem 2.2]{js}). Moreover,  Prunaru gave  a well-known  noncommutative Nahari theorem in \cite[Theorem 4.3]{pru}. That is, for any $A\in\mathcal M$,
$dist(A,\mathfrak A)=\|(I-P_0)AP_0\|$. Motivated by this well-known distance formula,  we may have the following symmetric distance formula.
\begin{proposition}\label{p1} Let $\mathfrak A$ be a maximal subdiagonal algebra. Then for any $A\in\mathcal M$,
\begin{equation}\label{d1}
\mbox{dist}(A,\mathfrak A)=\|(I-P_0)AP_0\|=\|(I-P_1)AP_1\|.
\end{equation}
\end{proposition}
\begin{proof} It is well-known that $\mbox{dist}(A,\mathfrak A)=\|(I-P_0)AP_0\|$ from \cite[Theorem 4.3]{pru}.  Now put $\tilde{P_0}=I-P_1$. Then $\tilde{P_0}$ is the orthogonal projection from $L^2(\mathcal M)$ onto $(H^2)^*$. If we consider $\mathfrak A^*$ as a maximal subdiagonal algebra with respect to $\Phi$, then $(H^2)^*$ is just the noncommutative $H^2$ space associated with $\mathfrak A^*$. It follows  from  this distance formula again,
$dist(A^*, \mathfrak A^*)=\|(I-\tilde{P_0})A^*\tilde{P_0}\|$. Thus
\begin{equation*}\label{d2}
dist(A, \mathfrak A)=dist(A^*,\mathfrak A^*)=\|(I-\tilde{P_0})A^*\tilde{P_0}\|=\|\tilde{P_0}A(I-\tilde{P_0})\|=\|(I-P_1)AP_1\|.
\end{equation*}
\end{proof}
Next we assume that $P\in \{P_0, P_1\}$, that is, $P=P_0$ or $P=P_1$. The following lemma is motivated by \cite[Theorem 4.3]{arv3}.
\begin{lemma}\label{l1}
Let $\mathfrak A$ be a maximal subdiagonal algebra and   $U\in \mathfrak A$  a partial isometry with $U^*U\in\mathfrak D$. If there exists a positive number $\varepsilon$ such that
$$\|(I-P)Ux\|\geq \varepsilon \|(I-P)x\|, \ \ \ \forall x\in U^*U(L^2(\mathcal M)),$$ then there exists a $B\in\mathfrak A$ such that $BU=U^*U$ and $\|B\|\leq 4\varepsilon^{-2}$.
\end{lemma}
\begin{proof}
We may assume that $P=P_0$ and  $\varepsilon<1.$  Since  $(I-P)UP=0$,  we have  $(I-P)U=(I-P)U(I-P)$.  On the one hand, $\forall x\in U^*U(L^2(\mathcal M)),$
\begin{align}\label{d3}
\|PU(I-P)x\|^2&= \|U(I-P)x\|^2-\|(I-P)U(I-P)x\|^2\\
\nonumber &\leq  \|U(I-P)x\|^2-\varepsilon^2 \|(I-P)x\|^2\\
\nonumber &\leq \|(I-P)x\|^2-\varepsilon^2 \|(I-P)x\|^2\\
\nonumber &=(1-\varepsilon^2)\|(I-P)x\|^2\\
\nonumber &\leq (1-\varepsilon^2)\|x\|^2.
\end{align}
On the other hand,
  Since $PU^*U=U^*UP$, we have  that
  \begin{equation}\label{d4}
  PU(I-P)x=PU(I-P)(I-U^*U)x=PU(I-U^*U)(I-P)x=0
  \end{equation} for all $x\in (I-U^*U)(L^2(\mathcal M)).$
Thus, $\forall x\in L^2(\mathcal M)$, $x=U^*Ux+(I-U^*U)x$ and
\begin{align}\label{d5}
&\ \ \ \ \|P(U(I-P)x\|^2=\|PU(I-P)U^*Ux\|^2\\
\nonumber &\leq (1-\varepsilon ^2)\|U^*Ux\|^2\leq (1-\varepsilon ^2)\|x\|^2
\end{align}
by (\ref{d3}) and (\ref{d4}). This implies that $\|PU(I-P)\|\leq (1-\varepsilon^2)^{\frac12}<1$. Thus
$$dist(U^*,\mathfrak A)=\|(I-P)U^*P\|\leq (1-\varepsilon^2)^{\frac12}<1$$ by Proposition \ref{p1}. It follows that there exists an operator  $C\in \mathfrak A$ such that
$\|U^*-C\|<(1-\varepsilon^2)^{\frac12}<1$. This implies that
$\|U^*-U^*UC\|=\|U^*UU^*-U^*UC\|\leq\|U^*-C\|$. Thus we may assume that $C=U^*UC\in\mathfrak A$ with $\|C\|\leq \|U^*\|+\|U^*-C\|\leq 1+(1-\varepsilon^2)^{\frac12}$. Put $T=(I-U^*U)+CU\in\mathfrak A$. Then
$\|I-T\|^2=\|U^*U-CU\|^2\leq \|U^*-C\|^2\|U\|^2<1-\varepsilon^2<1$. Hence $T^{-1}\in\mathfrak A$ with $\|T^{-1}\|\leq (1-(1-\varepsilon^2)^{\frac12})^{-1}$. Now put $B=T^{-1}C$. Then $\|B\|\leq \|T^{-1}\|\|C\|\leq (1+(1-\varepsilon^2)^{\frac12})(1-(1-\varepsilon^2)^{\frac12})^{-1}\leq 4\varepsilon^{-2}$.
Note that $T=I-U^*U+U^*UCUU^*U$. Then  $U^*UCUU^*U$ is invertible on $U^*U(L^2 (\mathcal M))$ and  $T^{-1}=I-U^*U+(U^*UCUU^*U)^{-1}$.
 It follows  that
 \begin{align*}
 \ \ \ \ BU&=T^{-1}CU=(I-U^*U+(U^*UCUU^*U)^{-1})CU\\
 &=(I-U^*U+(U^*UCUU^*U)^{-1})(U^*UCUU^*)\\
 &=(U^*UCUU^*U)^{-1}(U^*UCUU^*)=U^*U.\end{align*}
\end{proof}
We now give an interpolation theorem for a subdiagonal algebra with the universal factorization property. For any $\{A_k:1\leq k\leq N\}\subseteq \mathfrak A$, we always put $\alpha=\max\{\|A_k\|:1\leq k \leq N\}$.
\begin{theorem}\label{t1}
Let $\mathfrak A$ be a subdiagonal algebra with the universal factorization property and   $\{A_k: 1\leq k\leq N\}\subseteq \mathfrak A$.
If there exists a positive number $\varepsilon$ such that
\begin{equation}\label{th0} \sum\limits_{k=1}^N\|A_kx\|^2\geq \varepsilon^2 \|x\|^2\end{equation} and
\begin{equation}\label{th1}
\sum\limits_{k=1}^N\|(I-P)A_kx\|^2\geq \varepsilon^2 \|(I-P)x\|^2
\end{equation} for all $x\in L^2(\mathcal M)$, then there exist operators $\{B_k:1\leq k\leq N\}\subseteq\mathfrak A$
 such that
 \begin{equation}\label{th2}
 \sum\limits_{k=1}^NB_kA_k=I
 \end{equation} with $\|B_k\|\leq 4 N\alpha \varepsilon^{-3}$ for $1\leq k\leq N$.
\end{theorem}
\begin{proof}
Without loss of generality, we may assume that $P=P_0$,  $\alpha\leq 1$   and $0<\varepsilon<1$. By (\ref{th0}),  we have
$\sum\limits_{k=1}^N A_k^*A_k\geq \varepsilon^2 I$. That is, $\sum\limits_{k=1}^N A_k^*A_k\in\mathcal M$ is an invertible operator. Hence  $\sum\limits_{k=1}^N A_k^*A_k=C^*C$ for an invertible operator $C\in\mathfrak A$ by the universal factorization property of $\mathfrak A$. Thus
 $\varepsilon^2I\leq C^*C\leq N$. This implies that  that $\|C\|\leq N^{\frac12}$  and  $\|C^{-1}\|
\leq \varepsilon ^{-1}$.  Put $C_k=A_kC^{-1}\in\mathfrak A$, $1\leq k\leq N$. Then
$\sum\limits_{k=1}^N C_k^*C_k=I$.  Moreover, $(I-P)T=(I-P)T(I-P)$,  $\forall T\in \mathfrak A$. This implies that
$I-P=(I-P)CC^{-1}=(I-P)C(I-P)C^{-1}$. Thus
\begin{align}\label{th3}
 \|I-P)x\|^2\leq\|(I-P)C\|^2\|(I-P)C^{-1}x\|^2
\leq N \|(I-P)C^{-1}x\|^2
\end{align} for any $x\in L^2(\mathcal M)$.
It follows that $\|(I-P)C^{-1}x\|^2\geq N^{-1}\|(I-P)x\|^2$ and
\begin{align}\label{th4}
&\ \ \ \ \sum\limits_{k=1}^N\|(I-P)C_kx\|^2
 =  \sum\limits_{k=1}^N\|(I-P)A_kC^{-1}x\|^2\\
 \nonumber &\geq   \varepsilon^2 \|(I-P)C^{-1}x\|^2
 \geq \varepsilon^2 N^{-1}\|(I-P)x\|^2
\end{align} for all $x\in L^2(\mathcal M)$ by (\ref{th1}).
Now  put $\mathcal N=\mathcal M\otimes M_N$ and $\mathcal A=\mathfrak A\otimes M_N$, where $M_N$ is the algebra of all complex $N\times N$ matrices. Then $\mathcal A\subseteq \mathcal N$  is a maximal subdiagonal algebra with respect to $\Phi\otimes I$ with diagonal $\mathcal D=\mathfrak D\otimes M_N$.  It is elementary that $L^2(\mathcal N)=L^2(\mathcal M)\otimes L^2(M_N)$ and $H^2(\mathcal A)=H^2\otimes L^2(M_N)$. Put $Q=P\otimes I$ and
$$U=\left(\begin{array}{cccc}
C_1&0&\cdots&0\\
  C_2& 0 &\cdots&0\\
  \cdots&\cdots& \cdots& \cdots\\
 C_N&0&\cdots&0 \end{array}\right)\in \mathcal A.$$ Then $U$ is a partial isometry with
 $$U^*U=\left(\begin{array}{cccc}
I&0&\cdots&0\\
  0& 0 &\cdots&0\\
  \cdots&\cdots& \cdots& \cdots\\
 0&0&\cdots&0 \end{array}\right)=I\otimes E_{11} \in \mathcal D.$$  This implies that $U^*U(L^2(\mathcal N))=L^2(\mathcal M)\otimes E_{11}$. Moreover, $I-Q=(I-P)\otimes I$ and $\forall x\in L^2(\mathcal M)$,
 $ \tilde{x}=x\otimes E_{11}\in U^*U(L^2(\mathcal N))$. Thus
 \begin{align}\label{th5}
& \|(I-Q)U\tilde{x}\|^2=\sum\limits_{k=1}^N\|(I-P)C_kx\|^2\\
 \nonumber &\geq \varepsilon^2 N^{-1}\|(I-P)x\|^2=\varepsilon^2 N^{-1}\|(I-Q)\tilde{x}\|^2.
 \end{align} By Lemma \ref{l1}, there exists an operator  $V\in \mathcal A$ such that $VU=U^*U=I\otimes E_{11}$ with $\|V\|\leq 4N\varepsilon^{-2}$.
 Note that  $V=(V_{jk})\in \mathcal A$ for some $V_{jk}\in\mathfrak A$. Then $\sum\limits_{k=1}^NV_{1k}C_k=I$. Put $B_k=C^{-1}V_{1k}\in \mathfrak A$ for
 $1\leq k\leq N$. Then $\|B_k\|\leq \|V_{1k}\|\|C^{-1}\|\leq \|V\|\|C^{-1}\|\leq 4N\varepsilon^{-3}$  for $1\leq k\leq N$ and
 \begin{align*}
 &\ \ \ \ \sum\limits_{k=1}^NB_kA_k=\sum\limits_{k=1}^NC^{-1}V_{1k}A_{k}C^{-1}C\\
 &=C^{-1}(\sum\limits_{N=1}^NV_{1k}C_k)C=I.
 \end{align*}
\end{proof}
 We recall that the left(resp. right) Toeplitz operator $T_A$(resp. $t_A$) on $H^2$ for any $A\in \mathcal M$ is defined by $T_Ax=P_0(Ax)$(resp. $t_Ax=P_0(xA)$), $\forall x\in H^2$.  As in the classical case, we easily have that $(T_A)^*=T_{A^*}$(resp. $(t_A)^*=t_{A^*}$) for any $A\in \mathcal M$. In particular, if $A\in \mathfrak A$, $T_A$(resp. $t_A$) is said to be  a left(resp. right) analytic Toeplitz operator.
 If we consider $\mathfrak A^*$ as a maximal subdiagonal algebra, and $\tilde{P_0} $ and $\tilde{P_1}$ are projections onto $(H^2)^*$ and $(H_0^2)^*$ respectively, then $\tilde{P_0}=I-P_1$ and $\tilde{P_1}=I-P_0$. Thus  for any $A^*\in\mathfrak A^*$, $(I-\tilde{P_1})A^*=(I-\tilde{P_1})A^*(I-\tilde{P_1})=P_0A^*=P_0A^*P_0$. That is, the restriction of $(I-\tilde{P_1})A^*$ on $(I-\tilde{P_1})L^2(\mathcal M)=H^2$  is just the  Toeplitz operator  $T_{A^*}$ on $H^2$. Then we have the following corollary.
 \begin{corollary}
 Let $\mathfrak A$ be a subdiagonal algebra with the universal factorization property and $\{A_k: 1\leq k\leq N\}\subseteq \mathfrak A$.
If there exists a positive number $\varepsilon$ such that
\begin{equation*}\label{th0} \sum\limits_{k=1}^N \|A_k^*x\|^2\geq \varepsilon^2\|x\|^2, \forall x\in L^2(\mathcal M) \end{equation*} and
\begin{equation*}
\sum\limits_{k=1}^N\|T_{A_k}^*x\|^2\geq \varepsilon^2 \|x\|^2, \forall x\in H^2,
\end{equation*}  then there exist operators $\{B_k:1\leq k\leq N\}\subseteq\mathfrak A$
 such that
 \begin{equation*}\label{th2}
 \sum\limits_{k=1}^NA_kB_k=I
 \end{equation*} with $\|B_k\|\leq 4 N\alpha \varepsilon^{-3}$ for $1\leq k\leq N$.
 \end{corollary}
 Symmetrically, if we consider the right representation of $\mathcal M$ on $L^2(\mathcal M)$, we have the following result.
  \begin{corollary}
  Let $\mathfrak A$ be a subdiagonal algebra with the universal factorization property and $\{A_k: 1\leq k\leq N\}\subseteq \mathfrak A$.
If there exists a positive number $\varepsilon$ such that
\begin{equation*} \sum\limits_{k=1}^N \|xA_k\|^2\geq \varepsilon^2\|x\|^2, \forall x\in L^2(\mathcal M) \end{equation*} and
\begin{equation}\label{c251}
\sum\limits_{k=1}^N\|(I-P)(xA_k)\|^2\geq \varepsilon^2 \|(I-P)x\|^2, \forall x\in L^2(\mathcal M),
\end{equation}  then there exist operators $\{B_k:1\leq k\leq N\}\subseteq\mathfrak A$
 such that
 \begin{equation*}\label{th2}
 \sum\limits_{k=1}^NA_kB_k=I
 \end{equation*} with $\|B_k\|\leq 4 N\alpha \varepsilon^{-3}$ for $1\leq k\leq N$.
 \end{corollary}
 \begin{proof} Without loss of generality, we assume that $P=P_0$. Then
  $$\sum_{k=1}^N\|A_k^*x\|^2=\sum_{k=1}^{N}\|x^*A_k\|^2\geq\varepsilon^2\|x^*\|^2=\varepsilon^2\|x\|^2,
 \forall x\in L^2(\mathcal M).$$ On the other hand,
 For any  $x=z^*+d+y\in L^2(\mathcal M)$, where $y,z\in H_0^2$ and  $d\in L^2(\mathfrak D)$, we  have $x^*=y^*+d^*+z$ and  $(I-P_0)x^*=y^*$. Thus $((I-P_0)x^*)^*=y=(I-\tilde{P_0})x$. In particular, for any $A\in\mathfrak A$, $(I-\tilde{P_0})(A^*x)=\left((I-P_0)(A^*x)^*\right)^*=\left((I-P_0)(x^*A)\right)^*$. It follows from (\ref{c251}) that
  \begin{align*}\ \ \ \ \ \ \ \ \ &\sum_{k=1}^N\|(I-\tilde{P_0})A_k^*x\|^2=\sum_{k=1}^{N}\|\left((I-P_0)(x^*A_k)\right)^* \|^2
   =\sum_{k=1}^{N}\|(I-P_0)(x^*A_k) \|^2\\
  &\geq\varepsilon^2\|(I-P_0)x^*\|^2=\varepsilon^2\|\left((I-P_0)x^*\right)^*\|^2
  = \varepsilon^2\|(I-\tilde{P_0}) x\|^2.
  \end{align*}
  That is,  operators $\{A_1^*,A_2^*,\cdots, A_N^*\}\subseteq \mathfrak A^*$ satisfy the conditions of Theorem \ref{t1}. It follows that there exist
  $\{B_1^*,B_2^*,\cdots, B_N^*\}\subseteq \mathfrak A^*$ such that $\sum\limits_{k=1}^nB_k^*A_k^*=I$,  and consequently,
  $\sum\limits_{k=1}^NA_kB_k=I$.
 \end{proof}
 \section{Interpolation in nest subalgebras of von Neumann algebras}
We  give the following generalization of Arveson's interpolation theorem for nest algebras in $\mathcal B(\mathcal H)$ and get a noncommutative operator-theoretic variant of the Corona theorem  for nest subalgebras of a von Neumann algebra as in \cite{arv3} using Theorem \ref{t1}. In fact, this is precisely the   nest subalgebra component  of a type 1  subdiagonal algebra discussed in Section 4. Let $L$ be one of the following completely ordered lattices:

(1) $\mathbb Z\cup\{-\infty,+\infty\}=\{n:-\infty\leq n\leq \infty\}$;

(2) $\mathbb Z_+\cup\{+\infty\}=\{n:0\leq n\leq +\infty\}$;

(3) $\mathbb Z_-\cup\{-\infty\}=\{n:-\infty\leq n\leq 0\}$;

(4) $\{n: 0\leq n \leq K\}$ for a positive integer $K$.

  In fact, $L$ may be regarded as a sublattice of $(1)$.  Without loss of generality, we may assume that $L=\mathbb Z\cup \{-\infty, \infty\}$. Other cases are similar.  Let  $\mathcal N=\{Q_n: n\in L\}\subseteq \mathcal M$ be a  projection nest, that is, $Q_n< Q_{n+1} $  for any $n\in\mathbb Z$, $P_{-\infty}=\bigwedge_{n\in L}Q_n=0$ and  $P_{+\infty}=\bigvee_{n\in L} Q_n=I$.    Let $\mathfrak A=\mbox{alg}_{\mathcal M}\mathcal N=\{A\in \mathcal M: (I-Q_n)AQ_n=0, \forall n\in L\}$ be  the nest subalgebra of $\mathcal M$ associated with $\mathcal N$(cf.\cite{gil1}). Then  $\mathfrak A$ is  a maximal subdiagonal with respect to $\Phi$, where $\Phi(A)=\sum\limits_{-\infty}^{+\infty}E_nAE_n$ for all $A\in\mathcal M$ by \cite[Corollary 3.1.2]{arv3}. Put  $E_n=Q_n-Q_{n-1}$ for any   $n\in\mathbb Z$. Then $Q_n=\sum\limits_{m\leq n}E_m$ and $I-Q_n=\sum\limits_{m>n}E_m$. It is elementary that $\mathfrak A=\bigvee\{E_mAE_k: m\leq k,A\in\mathcal M\} $ and $\mathfrak A_0=\bigvee\{E_mAE_k:  m<k, A\in\mathcal M\}$.  Moreover,
  We still denote by $P_0$ and $P_1$ projections from $L^2(\mathcal M)$ onto $H^2$ and $H_0^2$ respectively.
\begin{theorem}\label{t2} Let  $\mathfrak A=\mbox{alg}_{\mathcal M}\mathcal N$ be  the nest subalgebra associated with the nest  $\mathcal N=\{P_n:n\in L\}$ in $\mathcal M$ and  $\{A_k: 1\leq k\leq N\}\subseteq \mathfrak A$. If there exists a positive number $\varepsilon$ such that
\begin{equation}\label{ar1}
\sum\limits_{k=1}^N\|(I-Q_n)A_kx\|^2\geq \varepsilon^2 \|(I-Q_n)x\|^2
\end{equation}
 for all $x\in\mathcal H$ and all $n\in L$, then there exist operators $\{B_k:1\leq k\leq N\}\subseteq\mathfrak A$
 such that
 \begin{equation}\label{ar2}
 \sum\limits_{k=1}^NB_kA_k=I
 \end{equation} with $\|B_k\|\leq 4 N\alpha \varepsilon^{-3}$ for $1\leq k\leq N$.
\end{theorem}
\begin{proof}  We assume that $L=\mathbb Z\cup\{-\infty,+\infty\}$.  Other cases are similar.  As in the case $\mathcal M=\mathcal B(\mathcal H)$(cf. \cite[Theorem 4.3]{arv3}),
 we  immediately have that $\mathfrak A_0$ contains no  nonzero idempotent.  Then by \cite[Theorem 3.8]{js}, $\mathfrak A$ has the universal factorization property. Note that    $L^2(\mathcal M)=\oplus_{k,m}E_kL^2(\mathcal M)E_m$.
               Without loss of generality, we may also represent any element $x\in L^2(\mathcal M)$ as a matrix form $x=\left(E_kxE_m\right)_{k,m}=\left(x_{k,m}\right)_{k,m}$. Thus
\begin{align}\label{ar3}
  H^2=P_0L^2(\mathcal M)
& =\bigvee \{ E_mxE_k: x\in L^2(\mathcal M),\forall m\leq k, \}\\
\nonumber &=\{x\in L^2(\mathcal M): E_mxE_k=0, \forall m>k\}
\end{align}
and
\begin{align}\label{ar4}
  (H_0^2)^*=(I-P_0)L^2(\mathcal M)
 &=\bigvee \{ E_mxE_k:  x\in L^2(\mathcal M), \forall m> k\}\\
\nonumber &=\{x\in L^2(\mathcal M): E_mxE_k=0, \forall m\leq k\}.
\end{align}
It is trivial that the inequality (\ref{ar1}) is equivalent to
\begin{equation}\label{ar5}
\sum\limits_{k=1}^N(I-Q_n)A_k^*(I-Q_n)A_k(I-Q_n)\geq \varepsilon^2 (I-Q_n)
\end{equation}  for all $n$. Thus    the inequality (\ref{ar1}) also  holds  replacing $x\in \mathcal H$ by $ x\in L^2(\mathcal M)$.
 That is,
 \begin{equation}\label{ar6}
\sum\limits_{k=1}^N\|(I-Q_n)A_kx\|^2\geq \varepsilon^2 \|(I-Q_n)x\|^2
\end{equation}
for all $x\in L^2(\mathcal M)$. By letting $n=-\infty$, we easily have that
$$\sum_{k=1}^n\|A_kx\|^2\geq \varepsilon^2\|x\|^2, \forall x\in L^2(\mathcal M).$$Now
for any $A\in\mathfrak A$,
\begin{align}\label{ar8}
&\ \ \ \  (I-Q_n)A=\sum_{n+1\leq k\leq m}E_kAE_m\\
\nonumber&=\left(\begin{array}{ccccc}\ddots&\vdots&\vdots&\vdots&\vdots\\
  \cdots& 0 &0&0&\cdots\\
  \cdots&0& A_{n+1,n+1}&A_{n+1,n+2}&\cdots\\
  \cdots&0&0&A_{n+2,n+2}&\cdots\\
  \vdots&\vdots&\vdots&\vdots&\ddots\end{array}\right).
\end{align}
Moreover, since $E_m\in\mathfrak D$,  we have  that $P_0R_{E_m}=R_{E_m}P_0$  for any $m$. Thus  $((I-P_0)x)E_m=R_{E_m}(I-P_0)x=(I-P_0)R_{E_m}x=(I-P_0)(xE_m)$ for any $m$ and all $x\in L^2(\mathcal M)$. In particular,
\begin{align}\label{ar9}
  &\ \ \ \ ((I-P_0)x)E_n=(I-P_0)(xE_n)\\
 \nonumber &=\left(\begin{array}{ccccc}\ddots&\vdots&\vdots&\vdots&\vdots\\
  \cdots& 0 &0&0&\cdots\\
  \cdots&x_{n+1,n}& 0&0&\cdots\\
  \cdots&x_{n+2,n}&0&0&\cdots\\
  \vdots&\vdots&\vdots&\vdots&\ddots\end{array}\right)\\
  \nonumber &=\oplus_{k=n+1}^{+\infty}  x_{k,n}
\end{align} in $L^2(\mathcal M)$.
It now follows from  (\ref{ar9}) that
\begin{equation}\label{ar10}
(I-Q_n)((I-P_0)(xE_n))=(I-P_0)(xE_n).
\end{equation}  Again, by (\ref{ar8}) and (\ref{ar10}), we have
\begin{align}\label{ar11}
&\ \ \ \ (I-Q_n)A(I-Q_n)((I-P_0)(xE_n))\\
\nonumber &=\left(\begin{array}{ccccc}\ddots&\vdots&\vdots&\vdots&\vdots\\
  \cdots& 0 &0&0&\cdots\\
  \cdots&\sum_{m=n+1}^{\infty}A_{n+1,m}x_{m,n}& 0&0&\cdots\\
  \cdots& \sum_{m=n+2}^{\infty}A_{n+2,m}x_{m,n}&0&0&\cdots\\
  \vdots&\vdots&\vdots&\vdots&\ddots\end{array}\right)\\
\nonumber &=\bigoplus_{k=n+1}^{+\infty} \left(\sum_{m=k}^{\infty}A_{k,m} x_{m,n}\right)
\end{align} in $L^2(\mathcal M)$.
It follows   from (\ref{ar11}) that
\begin{equation}\label{ar12}
(I-P_0)(I-Q_n)A((I-Q_n)(I-P_0)(xE_n))=(I-Q_n)A((I-Q_n)(I-P_0)(xE_n))\end{equation} for all $n\in\mathbb Z$.
On the other hand,
\begin{align}\label{ar13}
&\ \ \ \ A(I-P_0)(xE_n)\\
\nonumber &=\left(\begin{array}{ccccc}\ddots&\vdots&\vdots&\vdots&\vdots\\
  \cdots& \sum_{m=n+1}^{\infty}A_{n,m}x_{m,n} &0&0&\cdots\\
  \cdots&\sum_{m=n+1}^{\infty}A_{n+1,m}x_{m,n}& 0&0&\cdots\\
  \cdots& \sum_{m=n+2}^{\infty}A_{n+2,m}x_{m,n}&0&0&\cdots\\
  \vdots&\vdots&\vdots&\vdots&\ddots\end{array}\right)\\
  \nonumber &=\bigoplus_{k=-\infty}^{n} \left(\sum_{m=n+1}^{\infty}A_{k,m} x_{m,n}\right)\oplus \bigoplus_{k=n+1}^{+\infty} \left(\sum_{m=k}^{\infty}A_{k,m} x_{m,n}\right).
\end{align}
It follows from (\ref{ar9}) and (\ref{ar13})
that
\begin{align}\label{ar14}
&\ \ \ \ (I-P_0)A(I-P_0)(xE_n)=\bigoplus_{k=n+1}^{+\infty} \left(\sum_{m=k}^{\infty}A_{k,m} x_{m,n}\right)\\
\nonumber &=(I-Q_n)A(I-Q_n)(I-P_0)(xE_n).
\end{align}
Moreover,
\begin{equation}\label{ar15}
(I-P_0)A(I-P_0)x=\oplus_{n=-\infty}^{+\infty} (I-P_0)A(I-P_0)(xE_n).
\end{equation}
Thus
\begin{align}\label{ar16}
&\ \ \ \  \|(I-P_0)A(I-P_0)x\|^2=\sum_{n=-\infty}^{+\infty}\| (I-P_0)A(I-P_0)(xE_n)\|^2\\
\nonumber &=\sum_{n=-\infty}^{+\infty}\|(I-Q_n)A(I-Q_n)(I-P_0)(xE_n)\|^2.
\end{align}
Furthermore, by (\ref{ar10}), we have
\begin{align}\label{ar17}
&\ \ \ \  \sum\limits_{k=1}^N\|(I-P_0)A_kx\|^2\\
\nonumber &=\sum\limits_{k=1}^N\|(I-P_0)A_k(I-P_0)x\|^2\\
\nonumber &=\sum\limits_{k=1}^N\sum_{n=-\infty}^{+\infty}\| (I-P_0)A_k(I-P_0)(xE_n)\|^2\\
\nonumber &=\sum\limits_{k=1}^N\sum_{n=-\infty}^{+\infty}\|(I-Q_n)A(I-Q_n)(I-P_0)(xE_n)\|^2\\
\nonumber &\geq  \sum_{n=-\infty}^{+\infty}\varepsilon^2 \|(I-Q_n)(I-P_0)(xE_n)\|^2\\
\nonumber &=\varepsilon^2 \sum_{n=-\infty}^{+\infty} \|(I-P_0)(xE_n)\|^2\\
\nonumber &=\varepsilon^2\|(I-P_0)x\|^2.
\end{align} It now follows from  Theorem \ref{t1} that  there exist $\{B_k:1\leq k\leq N\}\subseteq \mathfrak A$  with $\|B_k\|\leq 4N\alpha \varepsilon^{-3}$ such that
$\sum_{k=1}^NB_kA_k=I$.
\end{proof}
 To establish  a noncommutative   operator-theoretic variant of the Corona theorem(cf. \cite[Theorem 6.3]{arv3}),  we firstly consider a special case for   nest subalgebras. We denote by $\mathcal M^{\prime}$ and  $\mathcal Z(\mathcal M)=\mathcal M\cap \mathcal M^{\prime}$ the commutant  and the center of $\mathcal M$ respectively.  If $\mathcal Z(\mathcal M)=\mathbb C I$, then we say that $\mathcal M$ is a factor. For any $A\in\mathcal M$, $c(A)$ denotes the central carrier of $A$.
 \begin{lemma}\label{l27}
 Let $\mathcal N=\{Q_j: 0\leq j\leq n\}\subseteq \mathcal M$ be a finite nest such that $0=Q_0<Q_1<\cdots<Q_n=I$ and $ \mathfrak A=\mbox{alg}_{\mathcal M}\mathcal N$ the nest subalgebra of $\mathcal M$  associated with $\mathcal N$. Assume that $\{A_k:1\leq k\leq N\}\subseteq \mathfrak A$.
 If there exists a positive number $\varepsilon$ such that
\begin{equation}\label{fl21}
\sum\limits_{k=1}^N\|(I-P_1)A_kx\|^2\geq \varepsilon^2 \|(I-P_1)x\|^2, \forall x\in L^2(\mathcal M),
\end{equation}  then
\begin{equation}\label{fl22}
\sum\limits_{k=1}^N\|A_kx\|^2\geq \varepsilon^2 \|x\|^2, \forall x\in L^2(\mathcal M).
\end{equation}
 \end{lemma}
\begin{proof}   As in the proof of Theorem \ref{t2}, put $E_j=Q_j-Q_{j-1}$, $j=1,2,\cdots,n$. Then $\Phi(A)=\sum\limits_{j=1}^nE_jAE_j$, $\forall A\in\mathcal M$. We easily have that $\mathfrak A=\vee\{E_jAE_k: A\in\mathcal M, 1\leq j\leq k\leq n\}$ and $\mathfrak A_0=\vee\{E_jAE_k:A\in\mathcal M, 1\leq j<k\leq n\}$. At the same time,
$L^2(\mathcal M)=\oplus_{j,k=1}^n E_jL^2(\mathcal M)E_k$, $H^2=\oplus_{1\leq j\leq k\leq n} E_jL^2(\mathcal M)E_k$ and
$H^2_0=\oplus_{1\leq j< k\leq n} E_jL^2(\mathcal M)E_k$. It now follows that $P_1=\sum_{1\leq j<k\leq n}L_{E_j}R_{E_k}$ and $I-P_1=\sum_{1\leq k\leq j\leq n}L_{E_j}R_{E_k}$.  Note that $\mathcal Z(\mathcal M)\subseteq \mathfrak D$. Thus for any projection $C\in \mathcal Z(\mathcal M)$, $C\mathfrak A$ is the  nest subalgebra of $C\mathcal M$ with nest $C\mathcal N=\{CQ_k:1\leq k\leq n\}$.  For any $1\leq k\leq n$, we prove that  the inequality (\ref{fl22}) holds for $xQ_k$ by induction.
If  $k=1$, then $Q_1=E_1$ and   $xE_1\in(I-P_1)L^2(\mathcal M)$ for any $x\in L^2(\mathcal M)$. It follows that
\begin{align*} &\ \ \ \ \sum_{j=1}^N\|A_jxQ_1\|^2=\sum_{j=1}^n\|(I-P_1)A_jxE_1\|^2\\
&\geq \varepsilon^2\|(I-P_1)xE_1\|^2=\varepsilon^2\|xE_1\|^2=\varepsilon^2\|xQ_1\|^2\end{align*}
by inequality (\ref{fl21}). Assume that for all $j\leq k-1$, the  inequality (\ref{fl22}) holds for all $xQ_j$. Then for any $xQ_k$, we have   $xQ_k=xQ_{k-1}\oplus xE_k$ and $AxQ_k=AxQ_{k-1}\oplus AxE_k$.  We next claim that inequality (\ref{fl22}) holds for  $xE_k$.  Recall that  $c(A)$  is the central carrier of $A$ for any $A\in\mathcal M$.  Put $C_1=c(Q_{k-1})c(E_k)$ and $C_2=c(E_k)-C_1$. Then $c(E_k)=C_1+C_2$,
$c(C_1Q_{k-1})=C_1=c(C_1E_k)$ and $C_2Q_{k-1}=0$. Take a maximal family of partial isometries $\{V_l:l\geq1\}\subseteq C_1E_k\mathcal MC_1Q_{k-1}$ such that $V_r^*V_s=0$ for any $r\not=s$. Then $\sum_{l\geq1}V_lV_l^*=C_1E_k$. Otherwise, since $c(C_1P_{k-1})=C_1=c(C_1E_k)$, there  is a nonzero partial isometry $V_0$ such that $V_0^*V_0\leq C_1Q_{k-1}$ and $0\not=V_0V_0^*\leq C_1(E_k-\sum_{l\geq1}V_lV_l^*)$. This contradicts with the  maximality of $\{V_l:l\geq1\}$. Now for any $l\geq1$, we have   $xE_kV_l=xE_kV_lQ_{k-1}$. By induction assumption,
$$\sum\limits_{j=1}^N\|A_jxE_kV_l\|^2\geq \varepsilon^2\|xE_kV_l\|^2.$$
This means that
\begin{align*}
&\ \ \ \ \sum\limits_{j=1}^N\|A_jxE_kV_l\|^2=\sum\limits_{j=1}^N \mbox{tr}(V_l^*E_kx^*A_j^*A_jxE_kV_l)
=\sum\limits_{j=1}^N \mbox{tr}(E_kx^*A_j^*A_jxE_kV_lV_l^*)\\
&\geq\varepsilon^2\|xE_kV_l\|^2=\varepsilon^2\mbox{tr}(V_l^*E_kx^*xE_kV_l)
=\varepsilon^2\mbox{tr}(E_kx^*xE_kV_lV_l^*)
\end{align*}
for all $l\geq1$. Thus
\begin{align}\label{fl23}
&\ \ \ \ \sum\limits_{j=1}^N\|A_jxC_1E_k\|^2= \sum\limits_{j=1}^N \mbox{tr}(C_1E_kx^*A_j^*A_jxE_kC_1 )\\
\nonumber &=\sum\limits_{j=1}^N\sum_{l\geq1} \mbox{tr}(C_1E_kx^*A_j^*A_jxC_1E_kV_lV_l^*)\\
\nonumber&=\sum_{l\geq1}\sum\limits_{j=1}^N  \mbox{tr}(C_1E_kx^*A_j^*A_jxC_1E_kV_lV_l^*)\\
\nonumber&\geq \varepsilon^2\sum_{l\geq1}\mbox{tr}(C_1E_kx^*xC_1E_kV_lV_l^*)\\
\nonumber&=\varepsilon^2 \mbox{tr}(C_1E_kx^*xC_1E_k)\\
\nonumber&=\varepsilon^2\|xC_1E_k\|^2.
\end{align}
On the other hand, $C_2Q_{k-1}=0$. This implies  that $C_2\mathcal N=\{C_2Q_k,\cdots, C_2Q_n\}$ and $C_2\mathfrak A$ is the nest subalgebra $C_2\mathcal M\cap\mbox{alg}C_2\mathcal N$. Thus $(I-P_1)C_2L^2(\mathcal M)=\vee\{C_2E_rx E_s:  x\in L^2(\mathcal M), k\leq r\leq s\leq n\}$. In particular,
$C_2L^2(\mathcal M)E_k\subseteq (I-P_1)C_2L^2(\mathcal M)\subseteq (I-P_1)L^2(\mathcal M)$. Therefore,
\begin{equation}\label{fl24}
\sum\limits_{j=1}^N\|A_jxC_2E_k\|^2\geq \varepsilon^2\|xC_2E_k\|^2.
\end{equation} Combining  (\ref{fl23})  and (\ref{fl24}), we have
\begin{align*}
&\ \ \ \ \sum\limits_{j=1}^N\|A_jxE_k\|^2=\sum\limits_{j=1}^N\|A_jx(C_1+C_2)E_k\|^2\\
&=\sum\limits_{j=1}^N\|A_jxC_1E_k\|^2+\sum\limits_{j=1}^N\|A_jxC_2E_k\|^2\\
&\geq\varepsilon^2\|xC_1E_k\|^2+\varepsilon^2\|xC_2E_k\|^2\\
&=\varepsilon^2\|xE_k\|^2.
\end{align*}
Consequently, \begin{align*}&\ \ \ \  \sum\limits_{j=1}^N\|A_jxQ_k\|^2=\sum\limits_{j=1}^N\|A_jx(Q_{k-1}+E_k)\|^2\\
&=\sum\limits_{j=1}^N\|A_jxP_{k-1}\|^2+\sum\limits_{j=1}^N\|A_jxE_k\|^2\\
&\geq \varepsilon^2(\|xQ_{k-1}\|^2+\|xE_k\|^2)\\
&=\varepsilon^2\|xQ_k\|^2.
\end{align*}
By taking  $k=n$, the inequality (\ref{fl22}) holds for any $x\in L^2(\mathcal M)$.
\end{proof}
\begin{theorem}\label{t28}
 Let $\mathcal N=\{Q_n:   n\in L\}\subseteq \mathcal M$ be a   nest  and $ \mathfrak A=\mbox{alg}_{\mathcal M}\mathcal N$ the nest subalgebra associated with $\mathcal N$. Assume that $\{A_k:1\leq k\leq N\}\subseteq \mathfrak A$.
 If there exists a positive number $\varepsilon$ such that
\begin{equation}\label{ft281}
\sum\limits_{k=1}^N\|(I-P_1)A_kx\|^2\geq \varepsilon^2 \|(I-P_1)x\|^2, \forall x\in L^2(\mathcal M),
\end{equation}  then there exist operators $\{B_k:1\leq k\leq N\}\subseteq\mathfrak A$
 such that
 \begin{equation*}
 \sum\limits_{k=1}^NB_kA_k=I
 \end{equation*}  with $\|B_k\|\leq 4N\alpha\varepsilon^{-3}$.
 \end{theorem}
\begin{proof}
Without loss of generality, we may assume that $L=\{n: -\infty\leq n\leq +\infty\}$ and  $Q_n<Q_{n+1}$ for all $n\in L$. We also may take $\alpha=1$.  For any positive integers $m$ and $n$,  the projection  $Q_{m,n}=Q_n-Q_{-m}\in \mathcal Z(\mathfrak D)$ is a semi-invariant projection of $\mathfrak A$ such that
$Q_{m,n}\mathfrak AQ_{m,n}\subseteq Q_{m,n}\mathcal MQ_{m,n}$ is a nest subalgebra  with  the  finite nest $\mathcal N_{m,n}=\{Q_{m,n}Q_k: -m\leq k\leq n\}$. It is elementary that $Q_{m,n}\mathfrak AQ_{m,n}\subseteq \mathfrak A$ is also a maximal  subdiagonal  algebra with respect to $\Phi|_{Q_{m,n}\mathcal MQ_{m,n}}$ with $\left(Q_{m,n}\mathfrak AQ_{m,n}\right)_0=Q_{m,n}\mathfrak A_0Q_{m,n}$.  Thus we have
 $$L^2(Q_{m,n}\mathcal MQ_{m,n})=Q_{m,n}L^2(\mathcal M)Q_{m,n},  \mbox{ and  } H^2_0(Q_{m,n}\mathfrak AQ_{m,n})=Q_{m,n}H^2_0(\mathfrak A)Q_{m,n},$$ that is,  $Q_{m,n}P_1$ is just the projection $P_{(m,n)1}$ from $L^2(Q_{m,n}\mathcal MQ_{m,n})$ onto $H^2_0(Q_{m,n}\mathfrak AQ_{m,n})$. In particular, $P_{(m,n)1}=P_1|_{Q_{m,n}L^2(\mathcal M)Q_{m,n}}$.  Let $A_k(m,n)=Q_{m,n}A_kQ_{m,n}$,  $1\leq k\leq N$. Then $A_k(m,n)\in Q_{m,n}\mathfrak AQ_{m,n}$ for any $1\leq k\leq N$. By (\ref{ft281}), we have
 \begin{equation*}
\sum\limits_{k=1}^N\|(I-P_{(m,n)1})A_k(m,n)x\|^2\geq \varepsilon^2 \|(I-P_{(m,n)1})x\|^2, \forall x\in Q_{m,n}L^2(\mathcal M)Q_{m,n}.
\end{equation*}
It follows  from Lemma \ref{l27} that
\begin{equation}\label{ft282}
\sum\limits_{k=1}^N\|Q_{m,n}A_kQ_{m,n}x\|^2\geq \varepsilon^2 \|x\|^2, \forall x\in Q_{m,n}L^2(\mathcal M)Q_{m,n}.
\end{equation}
 Thus by (\ref{ft282}),
\begin{align}\label{ft283}
 \sum\limits_{k=1}^N\|A_kQ_{m,n}x\|^2\geq  \sum\limits_{k=1}^N\|Q_{m,n}A_kQ_{m,n}x\|^2
\geq \varepsilon^2 \|Q_{m,n}x\|^2, \forall x\in L^2(\mathcal M).
\end{align}
  Since  $\lim\limits_{m,n\to+\infty}Q_{m,n}x=x$  for any $x\in L^2(\mathcal M)$,   we have
  \begin{align}\label{ft284}
 \sum\limits_{k=1}^N\|A_k x\|^2
\geq \varepsilon^2 \|x\|^2, \forall x\in L^2(\mathcal M)
\end{align} by letting $m,n\to+\infty$ in (\ref{ft283}). By (\ref{ft281}) and (\ref{ft284}), operators $\{A_k:1\leq k\leq N\}$ satisfy conditions of  Theorem \ref{t1}, and therefore there exist $\{B_k:1\leq k \leq N\}\subseteq\mathfrak A$ with
$\|B_k\|\leq 4N\varepsilon^{-3}$ such that
$\sum\limits_{k=1}^NB_kA_k=I.$
\end{proof} We next give a noncommutative  operator-theoretic variant of the Corona theorem for this nest subalgebra.
\begin{corollary}\label{c29}
Assume that $\{A_1,A_2,\cdots, A_N\}\subseteq \mbox{alg}_{\mathcal M}\mathcal N$.
 If there exists a positive number $\varepsilon$ such that
\begin{equation}\label{c291}
\sum\limits_{k=1}^N\|T^*_{A_k}x\|^2\geq \varepsilon^2 \|x\|^2, \forall x\in H^2,
\end{equation}  then there exist operators $\{B_k:1\leq k\leq N\}\subseteq\mathfrak A$
 such that
 \begin{equation*}
 \sum\limits_{k=1}^NA_kB_k=I
 \end{equation*}  with $\|B_k\|\leq 4N\alpha\varepsilon^{-3}$.
\end{corollary}
\begin{proof}
Note that $(\mbox{alg}_{\mathcal M}\mathcal N)^*$ is also a nest subalgebra with nest $\{(I-Q_n):n\in L\}$. It is known that $I-\tilde{P_1}=P_0$.
Then  for any $x\in H^2=(I-\tilde{P_1})L^2(\mathcal M)$, \begin{align*}
&\ \ \ \ \sum\limits_{k=1}^N\|(I-\tilde{P_1})A_k^*x\|^2=\sum\limits_{k=1}^N\|(I-\tilde{P_1})A_k^*(I-\tilde{P_1})x\|^2
\\
&=\sum\limits_{k=1}^N\|P_0A_k^*x\|^2=\sum\limits_{k=1}^N\|T_{A_k}^*x\|^2\geq
\varepsilon^2\|x\|^2.\end{align*}
It follows from Theorem \ref{t28} that  there exist $\{B_k^*:1\leq k\leq N\}\subseteq (\mbox{alg}_{\mathcal M}\mathcal N)^*$
such that $\sum\limits_{k=1}^NB_k^*A_k^*=I$, that is, $\sum\limits_{k=1}^NA_kB_k=I$.\end{proof}
\begin{remark} \label{r35}
By symmetry, if we replace (\ref{ft281})   by (\ref{c251})  in Theorem \ref{t28}, then $\sum\limits_{k=1}^N A_kB_k=I$ for some $\{B_k:1\leq k \leq N\}\subseteq \mbox{alg}_{\mathcal M}\mathcal N$. Similarly, if we replace left Toeplitz operator $T_{A_k}^*$ by  right Toeplitz operator $t_{A_k}^*$ in Corollary \ref{c29}, we may also  have that  $\sum\limits_{k=1}^N B_kA_k=I$  for some $\{B_k:1\leq k \leq N\}\subseteq \mbox{alg}_{\mathcal M}\mathcal N$.
\end{remark}
\section{Type 1 subdiagonal algebras are analytic operator algebras with periodic flows}
 Assume that $\mathfrak A$ is a maximal  subdiagonal algebra. A  closed subspace $\mathfrak M\subseteq L^2(\mathcal M)$ is said to be  a right(resp. left) invariant subspace of $\mathfrak A$ if $\mathfrak M\mathfrak A\subseteq\mathfrak M$(resp. $\mathfrak A\mathfrak M\subseteq\mathfrak M$). If it is both right and left invariant, then it is said to be a two side invariant subspace. It is well known    that for any right invariant subspace $\mathfrak M$,  there exist type 1  and type 2 right invariant subspaces $\mathfrak M_1$ and $\mathfrak M_2$ such that
 $\mathfrak M=\mathfrak M_1\oplus^{col}\mathfrak M_2$. If $\mathfrak M$  is of     type 1, then $\mathfrak M=\oplus^{col}V_{k\geq1}H^2$  for a family of partial isometries
 $\mathcal V=\{V_k:k\geq 1\}$  with $V^*_kW_l=0$ for $k\not=l$ and $V_k^*V_k\in\mathfrak D$ for any $k,l\geq1$(cf. \cite[Theorem 2.1,  Corollary 2.3]{ble} and \cite[Theorem 2.3, Corollary 2.6]{lab}). We call such a family of partial isometries a column orthogonal  family of partial isometries. By symmetry, we may also get similar results for left invariant subspaces. We recall that a maximal subdiagonal algebra $\mathfrak A$  is of type 1 if any nonzero  right(resp. left) invariant subspace of $\mathfrak A$ in $H^2$ is of   type 1(cf.\cite[Definition 2.1]{ji3}).
 In particular,   there exists a column orthogonal  family of   partial isometries $\mathcal U=\{U_{n} : n \geq 1\}$ in $\mathcal M$   such that $H_{0}^{2}=\oplus_{n\geq 1}^{col}U_{i}H^{2}$ by \cite[formula (2.2)]{ji3}. Moreover, if  $\mathfrak A_0^n$ is the $\sigma$-weakly closed ideal generated by $\{A_1A_2\cdots A_n: A_j\in\mathfrak A_0\}$ of $\mathfrak A$, then  $\mathfrak A$ is of type 1 if and only if $\cap_{n\geq1}\mathfrak A_0^n=\{0\}$ by \cite[Theorem 3.1]{jj}.

  On the other hand, if $\alpha=\{\alpha_t:t\in\mathbb R\}$ is a periodic flow on $\mathcal M$ and   $H^{\infty}(\alpha)$ is  the analytic operator algebra determined by  $\alpha$(cf.\cite{loe,saito,solel1,solel2}), then $H^{\infty}(\alpha)$ is of type 1(cf. \cite[Remark 3.1]{ji4}). In this section,  we prove that the converse does hold.  We then  give a form decomposition  for a type 1 subdiagonal algebra and  establish a conditional   expectation from $\left(L(\mathfrak D)\right)^{\prime}$(resp. $\left(R(\mathfrak D)\right)^{\prime}$) onto
$R(\mathcal M)$(resp. $L(\mathcal M)$) for  a type 1 subdiagonal algebra with  the $(1,1)$-form,  which  is a key tool   to get a noncommutative operator-theoretical variant of the Corona theorem for these algebras in next section.

        We  note that the modular automorphism group $\sigma^{\varphi}=\{\sigma_t^{\varphi}:t\in\mathbb R\}$  associated with $\varphi$  can be be extended to a strongly continuous
one-parameter group of surjective isometries  on $L^p(\mathcal M)$, that is,  $\sigma^{\varphi}=\{\sigma_t^{\varphi}:t\in\mathbb R\}$ is also an automorphism group of  $L^p(\mathcal M)$. We refer readers to  \cite[Section 10]{kos3} for details.

As in \cite{ji3}, put $\mathfrak M_0=H^2$ and  $\mathfrak M_n=[H^2\mathfrak A_0^n]=[\mathfrak A_0^n H^2]$ for $n\geq1 $. It is known that
$\mathfrak M_n$ is two side invariant and also  $\sigma^{\varphi}$ invariant with $\mathfrak M_{n+1}=[\mathfrak M_n\mathfrak A_0]=[\mathfrak A_0\mathfrak M_n]$. Thus  $W_n=\mathfrak M_n\ominus \mathfrak M_{n+1}$ is  both right and left  wandering subspace of $\mathfrak M_n$ with $W_0=L^2(\mathfrak D)$. Let  $M_n=\{A\in\mathcal M: Ah_0^{\frac12}\in W_n\}$ for any $n\geq 0$. Then $M_0=\mathfrak D$ and $M_n$ is a $\sigma$-weakly closed subspace of $\mathfrak A_0$ for any $n\geq1$. It is elementary that $M_n$ is a $\sigma$-weakly closed  $\mathfrak D$ bi-module in $\mathcal M$.
Before presenting the following result, we first remark on an elementary fact: If $S\subseteq \mathcal M$ is a $\sigma$-weakly closed subspace  that is  invariant under $\sigma^{\varphi}$(i.e., $\sigma^{\varphi}_t(S)=S$, $\forall t\in\mathbb R$), then $S\cap \mathcal T$ is $\sigma$-weakly dense in $S$. Here, $\mathcal T$ denotes  the set of all entire elements in $\mathcal M$, defined as  those elements  $X\in\mathcal M$ for which the map $t\to\sigma_t^{\varphi}(X)$ admits an extension to an $\mathcal M$-valued entire
function on $\mathbb C$. In fact, for any $X\in S$, define $X_r=\sqrt{\frac{r}{\pi}}\int\limits_{\mathbb R}e^{-rt^2}\sigma_t^{\varphi}(X)dt$(for $r>0$)  as in formula $(2.2)$  of \cite{ji1}. Then $X_r\in S\cap  \mathcal T$, and $X_r\to X$ as $r\to \infty.$
\begin{lemma}\label{l41} Both $W_n$ and $M_n$ are  invariant under $\sigma^{\varphi}$  and  for any $n\geq1$,
 $$W_n=[M_nh_0^{\frac12}]=[M_nL^2(\mathfrak D)]=
[L^2(\mathfrak D)M_n]=[h_0^{\frac12}M_n].$$
\end{lemma}
\begin{proof}
 It is known  from  \cite[Section 10]{kos3} that   $\sigma^{\varphi}$ may be extended to a strongly continuous
one-parameter group of surjective  isometries on $L^p(\mathcal M)$. In particular, When $p=2$,   $ \sigma_t^{\varphi}$ is a unitary operator on $L^2(\mathcal M)$ for all $t\in\mathbb R$. Since both $\mathfrak A$ and $\mathfrak A_0$ are $\sigma^{\varphi}$  invariant by \cite[Theorem 2.4]{jos1}, we have   $\mathfrak M_n$ is $\sigma^{\varphi}$  invariant. It follows  that
$W_n $ is $\sigma^{\varphi}$ invariant for any $n\geq1$. Thus $M_n$ is also $\sigma^{\varphi}$ invariant for any $n\geq0$. By \cite[Proposition 2.6]{ji3},
\begin{align*}
W_n&=\vee\{U_{i_1}U_{i_2}\cdots U_{i_n}L^2(\mathfrak D): U_{i_k}\in\mathcal U \}\\
\nonumber &=\vee\{U_{i_1}U_{i_2}\cdots U_{i_n}\mathfrak Dh_0^{\frac12}: U_{i_k}\in\mathcal U \}.
\end{align*}
Then $\vee\{U_{i_1}U_{i_2}\cdots U_{i_n}\mathfrak D: U_{i_k}\in\mathcal U \}\subseteq M_n$. Hence, $[M_nh_0^{\frac12}]=W_n$. Note that $M_n$ is a bi-module of $\mathfrak D$ and
$M_n\cap\mathcal T$ is $\sigma$-weakly dense in $M_n$. We have $W_n=[M_nh_0^{\frac12}]=[M_nL^2(\mathfrak D)]=[L^2(\mathfrak D)M_n]=[h_0^{\frac12}M_n]$.
\end{proof}
For $n\geq 1$, since $\mathfrak A$ is of type 1, we have
\begin{equation}\label{41}
\mathfrak M_n=\oplus_{m\geq 1}^{col}U_{n,m}H^2
\end{equation} for a column orthogonal family  of partial isometries $\{U_{n,m}:m\geq1\}\subseteq \mathcal M$.  Note that $U_{1,m}=U_m$ for any  $m\geq1$.  In this case,
 we have
  \begin{equation}\label{42}
  W_n=\oplus_{m\geq1}^{col}U_{n,m}L^2(\mathfrak D).
  \end{equation}  Put $K=\sup\{n: \mathfrak M_n\not=\{0\}\}$. Then $K$ is finite or $+\infty$. For any $ n\geq1 $($n\leq K$ if $K<+\infty$ ),
   we define $\mathfrak M_{-n}=\mathfrak M_n^*$,   $M_{-n}=M_n^*$. Then $[M_{-n}h_0^{\frac12}]=[(H^2)^*(\mathfrak A_0^*)^{n}]\ominus [(H^2)^*(\mathfrak A_0^*)^{n+1}]$ is the wandering subspace of  $\mathfrak M_n^*$  for $\mathfrak A_0^*$. By considering  the type 1 subdiagonal algebra $\mathfrak A^*$, We again have that
 \begin{equation}\label{43}
\mathfrak M_{-n}=\oplus_{m\geq 1}^{col}U_{-n,m}(H^2)^*
\end{equation} for a column orthogonal family  of isometries $\{U_{-n,m}:m\geq1\}$ and
   \begin{equation}\label{44}
  W_{-n}=\oplus_{m\geq1}^{col}U_{-n,m}L^2(\mathfrak D).
  \end{equation}
   Now put $U_{0,1}=I$. Then $\mathfrak M_0=H^2=U_{01}H^2$ and for any integer $n\in\mathbb Z$($-K\leq n\leq K$ if $K<+\infty$), there exists a column orthogonal family of partial isometries $\{U_{n,m}: m\geq1\}$ such that
  $W_n=\oplus_{m\geq1}^{col}U_{n,m}L^2(\mathfrak D)$.  If $K<+\infty$, then we may assume that $W_n=\{0\}$ for any $|n|>K$. We now have
  \begin{equation}\label{45} H^2=\oplus_{n\geq0}W_n \mbox{ and }
L^2(\mathcal M)=\oplus_{n\in\mathbb Z} W_n .
\end{equation}
\begin{theorem}\label{t42} Let $\mathfrak A$ be a type 1 subdiagonal algebra. Then $\mathfrak A=\vee\{M_n:n\geq 0\}$ and
$\mathcal M=\vee\{M_n: n\in\mathbb Z\}$.
\end{theorem}
\begin{proof} Note that $\vee\{U_{i_1}U_{i_2}\cdots U_{i_n}\mathfrak D: U_{i_k}\in\mathcal U \}\subseteq M_n$.  By \cite[Theorem 2.7]{ji3}, $\mathfrak A=\vee\{M_n:n\geq 0\}$. The second follows from  the fact that $\mathfrak A+\mathfrak A^*$ is $\sigma$-weakly dense in $\mathcal M$.
\end{proof}
\begin{lemma}\label{l43}
Let $x\in W_n$ and  $x=U|x|$ the polar decomposition of $x$. Then $U\in M_n$ and $|x|\in L^2(\mathfrak D)$.
\end{lemma}
 \vskip10pt
 \begin{proof} We may assume that $n\geq 0$. Since $W_n$  is the wandering subspace of $\mathfrak M_n$, it is known that $|x|^2\in L^1(\mathfrak D)$ from \cite[Theorem 2.3]{lab}. Thus $|x|\in L^2(\mathfrak D)$. We also know that both $x$ and $|x|$ are right wandering vectors and  $U$ is a  partial isometry with initial subspace  $[|x|\mathcal M]$ and final subspace
 $[x\mathcal M]$. Now $[|x|\mathcal M]=[|x|\mathfrak A_0^*]\oplus [|x|\mathfrak D]\oplus [|x|\mathfrak A_0]$ and $[|x|\mathfrak A_0^*] \oplus [|x|\mathfrak A_0] \bot L^2(\mathfrak D)$. It follows that $L^2(\mathfrak D)\ominus[ |x|\mathfrak D]\subseteq [|x|\mathcal M]^{\bot}$. Thus
 $U\left(L^2(\mathfrak D)\ominus[ |x|\mathfrak D]\right)=\{0\}$ and $UL^2(\mathfrak D)=[x\mathfrak D]\subseteq W_n$. In particular,
 $Uh_0^{\frac12}\in W_n$. Consequently, $U\in M_n$.
  \end{proof}
  \begin{lemma}\label{l44}$\forall n,m\in\mathbb Z$, $M_nM_m\subseteq M_{n+m}$. In particular, $M_n^*\mathfrak DM_n\subseteq \mathfrak D$.
  \end{lemma}
  \begin{proof}By Lemma \ref{l41} and \cite[Proposition 2.6]{ji3},  $W_n=\vee\{U_{i_1}U_{i_2}\cdots U_{i_n}L^2(\mathfrak D): U_{i_k}\in\mathcal U \}=[M_nh_0^{\frac12}]=[h_0^{\frac12}M_n]$ for $n\geq0$. If $m,n \geq0$, then
   \begin{align*}
   W_{m+n}&=\vee\{U_{i_1}U_{i_2}\cdots U_{i_m}\cdots U_{i_{m+n}}L^2(\mathfrak D): U_{i_k}\in\mathcal U \}\\
   &=\vee\{U_{i_1}U_{i_2}\cdots U_{i_m}[U_{m+1}\cdots U_{i_{m+n}}L^2(\mathfrak D)]: U_{i_k}\in\mathcal U \}\\
   &=\vee\{U_{i_1}U_{i_2}\cdots U_{i_m}W_n: U_{i_k}\in\mathcal U \}\\
   &=\vee\{U_{i_1}U_{i_2}\cdots U_{i_m}[h_0^{\frac12}M_n]: U_{i_k}\in\mathcal U \}\\
   &=\vee\{[U_{i_1}U_{i_2}\cdots U_{i_m}L^2(\mathfrak D)]M_n: U_{i_k}\in\mathcal U \}\\
   &=[M_mM_nh_0^{\frac12}].
   \end{align*}
   This means that  $M_mM_n\subseteq M_{m+n}$.  We similarly  have  that $M_mM_n\subseteq M_{m+n}$ if $m,n\leq 0$. Now we may assume that  $n\geq -m\geq 0$. Since $U_lL^2(\mathfrak D)U_k^*\subseteq L^2(\mathfrak D)$ by \cite[Lemma 3.1]{ji4} for any $l,k\geq 0$, we have
   \begin{align*}
   &\ \ \ \ [M_nM_mL^2(\mathfrak D)]=[M_nL^2(\mathfrak D)M_m]\\
   &=[M_n[L^2(\mathfrak D)M_{-m}^*]]=[M_n[M_{-m}L^2(\mathfrak D)]^*]\\
   &=[M_n\left (\vee\{U_{i_1}\cdots U_{i_{-m}}L^2(\mathfrak D): U_{i_k}\in\mathcal U \}\right)^*]\\
   &=[M_n\left (\vee\{ L^2(\mathfrak D)U_{i_{-m}}^*\cdots U_{i_{1}}^*: U_{i_k}\in\mathcal U \}\right)]\\
   &=\vee\{M_n L^2(\mathfrak D)U_{i_{-m}}^*\cdots U_{i_{1}}^*: U_{i_k}\in\mathcal U \}\\
   &=\vee\{U_{j_1}\cdots U_{j_{n+m}}\left(U_{n+m+1}\cdots U_{j_n} L^2(\mathfrak D)U_{i_{-m}}^*\cdots U_{i_{1}}^*\right): U_{j_l}, U_{i_k}\in\mathcal U \}\\
   &\subseteq \vee\{U_{j_1}\cdots U_{j_{n+m}} L^2(\mathfrak D): U_{j_l}\in\mathcal U \}\\
   &\subseteq W_{n+m}=[M_{n+m}h_0^{\frac12}].
   \end{align*}
   Thus $M_{n}M_m\subseteq M_{n+m}$. Note that $M_0=\mathfrak D$ and $M_n$ is a $\mathfrak D$ bi-module. Then
   $M_n^*\mathfrak DM_n\subseteq M_n^*M_n\subseteq \mathfrak D$.
  \end{proof}
    \begin{corollary}\label{c45}
  Let $A\in M_n$ and $A=U|A|$ the polar decomposition of $A$. Then $|A|\in \mathfrak D$ and $U\in M_n$.
  \end{corollary}
\begin{proof}Note that $|A|^2=A^*A\in M_0=\mathfrak D$. Then $|A|\in\mathfrak D$. Moreover,
$$R(|A|)=|A|( L^2(\mathcal M))=|A|((H_0^2)^*)\oplus |A|(L^2(\mathfrak D))\oplus |A|(H_0^2). $$ We have   $L^2(\mathfrak D)\ominus [|A|(L^2(\mathfrak D))]\bot R(|A|)$ and hence $U(L^2(\mathfrak D)\ominus [|A|L^2(\mathfrak D)])=\{0\}$. Thus $UL^2(\mathfrak D)=U|A|L^2(\mathfrak D)=AL^2(\mathfrak D)\subseteq W_n$. Therefore, $U\in M_n$.
\end{proof}

For any $n\in\mathbb Z$, denote by  $E_n$ and $P_n$ the projections  from $L^2(\mathcal M) $ onto $W_n$ and $\oplus_{m\geq n}W_m$ respectively.
Then $I=\sum_{n\in\mathbb Z}E_n$ and  $P_n=\sum_{m\geq n}E_m$ for $n\in\mathbb Z$. It is clear that $P_n$ is the projection from $L^2(\mathcal M) $ onto $\mathfrak M_n$ for any $n\geq0$. In particular, $P_0$ and $P_1$ are  the same as defined in Section 2 and $P_n(L^2(\mathcal M))$ is a two side invariant  subspace of $\mathfrak A$  for any $n\in\mathbb Z$  from Lemma \ref{l44}.
 Moreover, we define
 \begin{equation}\label{fn}
F_0=I \mbox{ and } F_n= \sum_{m\geq 1}U_{n,m}U_{n,m}^*, \ \ \ \forall n\not=0.
 \end{equation}
 For any $n\geq 1$, since $\mathfrak M_{n}\supseteq \mathfrak M_{n+1}$,  we have   $F_n\geq F_{n+1}$. Similarly, $F_n\leq F_{n+1}$ for any $n<0$.   It is elementary that $F_n\in \mathfrak D$ for all $n\in\mathbb Z$ by formulae (\ref{42}), (\ref{44}) and Lemma \ref{l44}.
\begin{lemma}\label{l46}
$F_n\in \mathcal Z(\mathfrak D)$ and $M_n=\oplus_{m\geq 1}^{col}U_{n,m}\mathfrak D$  for any  $n\in\mathbb Z$.  Moreover, $F_m A=A$ for all $A\in M_n$ and  $n\geq m\geq0$ and $F_nB=B$ for all $B\in M_m$ and $m\leq n\leq 0$.
\end{lemma}
\begin{proof}  Let $n\geq1$. By (\ref{41}) and (\ref{43}),  $\mathfrak M_n= \oplus_{m\geq 1}^{col} U_{n,m}H^2$ and $W_n=\oplus_{m\geq1}^{col}U_{n,m}L^2(\mathfrak D)$.
Then $F_n \mathfrak M_n=\mathfrak M_n$ and $F_nW_n=W_n$. In particular, $F_nA=A,$ $\forall A\in M_n$. Note that $\mathfrak DM_n=M_n$ and $U_{n,m}\in M_n$. This implies that  $F_n DU_{n,m}=DU_{n,m}$ for any $m\geq1$ and  $D\in\mathfrak D$. Thus $F_n DU_{n,m}U_{n,m}^*=DU_{n,m}U_{n,m}^*.$ Therefore,
$$F_nDF_n=\sum_{m\geq1}F_n DU_{n,m}U_{n,m}^*=\sum_{m\geq1}DU_{n,m}U_{n,m}^*=DF_n$$   for all $D\in\mathfrak D$. This means that $F_n\in \mathfrak D^{\prime}$. Consequently, $F_n\in\mathcal Z(\mathfrak D).$ If $n<0$, we similarly have that  $F_n\in\mathcal Z(\mathfrak D)$.

 Take any $A\in M_n$. Since $U^*_{n,m}A\in \mathfrak D$ by Lemma \ref{l44}, we have
 $$A=F_nA=\sum_{m\geq1}^{col}U_{n,m}(U^*_{n,m}A)\in \oplus_{m\geq 1}^{col}U_{n,m}\mathfrak D.$$ The converse is clear. Moreover, since $F_m\geq F_n$ if $n\geq m\geq0$, we have   $F_mA=F_mF_nA=A$ for any $A\in M_n$.  The another case is similar.
\end{proof}

For any $t\in\mathbb R$, we define $W_t=\sum_{n\in\mathbb Z}e^{int}E_n$ and  $\widetilde{\alpha}_t(A)=W_tAW_t^*$, $\forall A\in \mathcal B(L^2(\mathcal M))$.
It is easily shown  that $\widetilde{\alpha}=\{\widetilde{\alpha}_t:t\in \mathbb R\}$ is a periodic  flow on $\mathcal B(L^2(\mathcal M))$ with the period $2\pi$.
\begin{lemma}\label{l47} For any $n\in\mathbb Z$, we have
$\widetilde{\alpha}_t(L_A)=e^{int}L_A$ for all $t\in \mathbb R$ and $A\in M_n$. In particular, $\tilde{\alpha}=\{\tilde{\alpha}_t:t\in\mathbb R\}$ is a periodic  flow of $L(\mathcal M)$.
\end{lemma}
\begin{proof}  Since $L_DE_n=E_nL_D$ for all $n\in\mathbb Z$ and $D\in\mathfrak D$, it is trivial that $\widetilde{\alpha}_t(L_D)=L_D$, $\forall t\in\mathbb R$. By Lemma \ref{l46}, it is sufficient to show that $ \widetilde{\alpha}_t(L_{U_{n,m}})=e^{int}L_{U_{n,m}}$ for all $n\in\mathbb Z$ and $m\geq1$. In fact, for any $B\in M_k$, $Bh_0^{\frac12}\in [M_kh_0^{\frac12}]=E_kL^2(\mathcal M)$, we have $L_{U_{n,m}}E_k(Bh_0^{\frac12})=U_{n,m}Bh_0^{\frac12}\in [M_{n+k}h_0^{\frac12}]=E_{n+k}L^2(\mathcal M)$ by Lemma \ref{l44}. Thus $L_{U_{n,m}}E_k(x)\in E_{n+k}L^2(\mathcal M)$ for any $x\in L^2(\mathcal M)$ and
\begin{align}\label{fl471}
E_lL_{U_{n,m}}E_kx=\left\{
\begin{array}{l}0,\ \ \ \ \ \ \ \  \ \ \ \  \  l\not=n+k;\\
L_{U_{n,m}}E_kx, \  l=n+k.
\end{array}\right.\end{align}
It follows that
\begin{align*}\widetilde{\alpha}_t(L_{U_{n,m}})(x)&=(\sum_{l=-\infty}^{+\infty}e^{ilt}E_l)L_{U_{n,m}}(\sum_{k=-\infty}^{+\infty}e^{-ikt}E_k)x\\
&=\sum_{l=-\infty}^{+\infty}\sum_{k=-\infty}^{+\infty}e^{i(l-k)t} E_lL_{U_{n,m}}E_kx\\
&=\sum_{k=-\infty}^{+\infty}e^{int} L_{U_{n,m}}E_kx\\
&=e^{int} L_{U_{n,m}}\sum_{k=-\infty}^{+\infty} E_kx\\
&=e^{int} L_{U_{n,m}}x
\end{align*}from formula (\ref{fl471}).
Hence $\widetilde{\alpha}_t(L_{U_{n,m}})=e^{int}L_{U_{n,m}}.$ This implies that $\tilde{\alpha}_t(L(M_n))=L(M_n)$,  and thus $\tilde{\alpha}_t(L(\mathcal M))=L(\mathcal M)$  for all $t\in
\mathbb R$ by Theorem \ref{t42}. Hence $\tilde{\alpha}$ is a periodic flow of $L(\mathcal M)$.
\end{proof}
  We now  define a periodic flow $\alpha=\{\alpha_t:t\in\mathbb R\}$  on $\mathcal M$ by setting
  $$L_{\alpha_t(A)}=\widetilde{\alpha}_t(L_A),  \ \forall A\in\mathcal M,\ \forall t\in\mathbb R.$$
   By Lemma \ref{l47}, this implies  $\alpha_t(A)=e^{int}A$ for all $A\in M_n$ and $t\in\mathbb R$. We then  show that $\mathfrak A$ is precisely the analytic operator algebra associated with $\alpha$.
\begin{theorem}\label{t48} Let $\mathfrak A$ be a type 1 subdiagonal algebra and let  $\alpha$ be defined as  above. Then we have
$\mathfrak A=H^{\infty}(\alpha)$.
\end{theorem}
\begin{proof} Put $\widetilde{M}_n=\{A\in\mathcal M: \alpha_t(A)=e^{int}A, \forall t\in\mathbb R\}$, $\forall n\in\mathbb Z$. It is known that
$\widetilde{M}_0$  is the fixed algebra of $\alpha$, $H^{\infty}(\alpha)=\vee\{\widetilde{M}_n:n\geq 0\}$ and $H_0^{\infty}(\alpha)=\vee\{\widetilde{M}_n:n\geq1\}$ by \cite[Proposition 2.4]{solel1}(cf. \cite[Proposition 2.3]{solel2}).  Then we have $\mathfrak D\subseteq \widetilde{M}_0$ by Lemma \ref{l47}.
Conversely, take any $A\in\widetilde{M}_0$. Then $\alpha_t(A)=A$, $\forall t\in \mathbb R$, which implies that $\widetilde{\alpha}_t(L_A)=(L_A)$, $\forall t\in\mathbb R$.  Thus $L_AE_n=E_nL_A$ for all $n\in\mathbb Z$. In particular, $L_AL^2(\mathfrak D)\subseteq L^2(\mathfrak D)$, that is, $A\in\mathfrak D$. Thus $\widetilde{M}_0=\mathfrak D$. On the other hand, we have
$M_n\subseteq \widetilde{M}_n$ for all $n\in \mathbb Z$ by Lemma \ref{l47} again.  It now follows that
$\mathfrak A_0=\vee\{M_n:n\geq1\}\subseteq \vee\{\widetilde{M}_n: n\geq1\}=H_0^{\infty}(\alpha)$. Moreover, $\varphi\circ\alpha_t(D)=\varphi(D)$ for all $D\in\mathfrak D$ and $\varphi\circ\alpha_t(A)=\varphi(e^{int}A)=0=\varphi(A)$ for all $A\in M_n$ and $n\not=0$. Thus $\varphi\circ \alpha_t=\varphi$ for all $t\in\mathbb R$.  It follows that  the unique faithful normal conditional expectation from $\mathcal M$ onto $\widetilde{M}_0=\mathfrak D$ determined by $\alpha$ is $\Phi$. Thus
  both $\mathfrak A$ and $H^{\infty}(\alpha)$ are  maximal subdiagonal algebras with respect to $\Phi$  with
  $\mathfrak A\subseteq H^{\infty}(\alpha)$(cf. \cite[Theorem 3.5]{loe}). Hence $\mathfrak A=H^{\infty}(\alpha)$.
\end{proof}
Theorem \ref{t48} says that  a type 1 subdiagonal algebra $\mathfrak A$ coincides with   an analytic operator algebra determined by  a periodic flow in a $\sigma$-finite von Neumann algebra. In particular, $M_n=\widetilde{M}_n$  for any $n\in\mathbb Z$ in Theorem \ref{t48}.  We now  recall some notions for analytic operator algebras for a periodic flows from \cite{solel1, solel2}.   Let $n\in\mathbb Z$.
For any $T\in (L(\mathfrak D))^{\prime}$, we define
$$\beta_n(T)=\sum_{m\geq1}L_{U_{n,m}}TL_{U_{n,m}^*}.$$
Then $\beta_n$ is well-defined such that  $L_{F_n}\beta_n(T)=\beta_n(T)L_{F_n}=\beta_n(T)$.
On the other hand,  since  $L(\mathcal M)^{\prime}=R(\mathcal M)$ and $\mathfrak M_n $ is also left invariant for $\mathfrak A$ for any $n\geq0 $,  we similarly have
\begin{equation*}
\mathfrak M_n=\oplus_{k\geq1}^{row}H^2V_{n,k}
\end{equation*}
 for a family of row orthogonal partial isometries $\{V_{n,k}:k\geq1\}\subseteq \mathcal M$. In fact, since
  $\mathfrak M_n^*=\mathfrak M_{-n}=\oplus _{k\geq1}^{col}U_{-n,k}H^2$, we may choose $V_{n,k}=U_{-n,k}^*$ for any   $k\geq1$. We also  may define $\eta_n(T)=\sum_{k\geq1}R_{V_{n,k}}TR_{V_{n,k}^*}$ for any $T\in (R(\mathfrak D))^{\prime}$ as in \cite{solel2}.
The following  properties  of $\beta_n(n\in\mathbb Z)$  are  given in  \cite{solel2}.  Symmetrical properties for $\{\eta_n:n\in\mathbb Z\}$  follow similarly.
\begin{proposition}\label{p49}(\cite[Lemm 2.4]{solel2}) Let $n,m\in\mathbb Z$.

(1) $\beta_n $ is a $*$-homomorphism from $(L(\mathfrak D))^{\prime} $ onto $L_{F_n}(L(\mathfrak D))^{\prime} $ and  a $*$-isomorphism from $L_{F_{-n}}(L(\mathfrak D))^{\prime} $ onto $L_{F_n}(L(\mathfrak D))^{\prime} $.

(2) For any projection $Q\in (L(\mathfrak D))^{\prime}$,
 we have
\begin{equation*}\beta_n(Q)=\sup\{L_UQL_{U^*}: U\in M_n \mbox{ is a partial isometry}\}.\end{equation*}

 (3)
 $\beta_n\beta_m(T)=L_{F_n}\beta_{n+m}(T)$, $\forall T\in (L(\mathfrak D))^{\prime}$.

(4) For any $T\in (L(\mathfrak D))^{\prime}$, $T\in R(\mathcal M)$ if and only if $\beta_n(T)=L_{F_n}T$ for all $n\in\mathbb Z$.
In particular, for any projection  $Q\in (L(\mathfrak D))^{\prime}$,
$Q\in R(\mathcal M)$ if and only if $\beta_n(Q)\leq Q$.

(5) $\beta_n(E_0)=E_n$, $\beta_n(E_m)=L_{F_n}E_{n+m}$ and $\beta_n(P_m)=L_{F_n}P_{n+m}.$
\end{proposition}
 \begin{corollary}\label{c410}
Let $\lim\limits_{m\to-\infty}F_m=F_-$ and $\lim\limits_{m\to+\infty}F_m=F_+$. Then we have $F_{\pm}\in \mathcal Z(\mathcal M)\cap \mathfrak D$.
\end{corollary}
\begin{proof} Since $F_n\in \mathcal Z(\mathfrak D)$,  we have $L_{F_n}\in (L(\mathfrak D))^{\prime}$. By Proposition \ref{p49}(3), $\beta_n(I)=L_{F_n}$ and $\beta_n\beta_m(I)=L_{F_n}\beta_{n+m}(I)$ for any $m$. This means that $\beta_n(L_{F_m})=L_{F_n}L_{F_{n+m}}$. By letting $m\to\pm \infty$, we get
$\beta_n(L_{F_{\pm}})=L_{F_n}L_{F_{\pm}}$ for any $n\in\mathbb Z$. By Proposition \ref{p49}(4), $L_{F_{\pm}}\in R(\mathcal M)=(L(\mathcal M))^{\prime}$, which implies that
$F_{\pm}\in \mathcal Z(\mathcal M)$.
\end{proof}
\begin{proposition}\label{p411} If $F_+=0$(resp. $F_-=0$), then there is a finite or infinite  nest $\mathcal N=\{Q_n: -\infty\leq n\leq 0\}\subseteq \mathcal M$(resp.
$\mathcal N=\{Q_n: 0\leq n\leq +\infty \}\subseteq \mathcal M$  such that $\mathfrak A= \mbox{alg}_{\mathcal M}\mathcal N$.
\end{proposition}

\begin{proof} Note that $F_n\mathfrak M_n=\mathfrak M_n$ for any $n\geq0$. Then $F_n\mathfrak A_0^n=\mathfrak A_0^n$. This implies that
$F_n(L^2(\mathcal M))\supseteq [\mathfrak A_0^nL^2(\mathcal M)]$ for any $n\geq0$. In particular,  $\cap_{n\geq 1}[\mathfrak A_0^nL^2(\mathcal M)]\subseteq F_+(L^2(\mathcal M))=\{0\}$.  Now let $Q_{n}$ be  the projection   from $L^2(\mathcal M)$ onto $[\mathfrak A_0^{-n}L^2(\mathcal M)]$ for any $n<0$ and $Q_0=I$. Then  $Q_n\in\mathcal M$ with $\bigwedge_{n\leq 0} Q_n=0$.  By  \cite[Theorem 3.2]{jos2}, $\mathfrak A$ is the nest subalgebra with nest $\mathcal N=\{Q_n: -\infty\leq n\leq 0\}$. If $F_k=0$ for some $k\geq1$, then $\mathcal N$ is finite.  We may have similar result when $F_-=0$.
\end{proof}
 Proposition \ref{p411} states that if either $F_+$ or $F_-$ vanishes, then the type 1 subdiagonal algebra $\mathfrak A$ reduces to  a nest subalgebra  of the form  discussed in Section 3. In  general, we  deduce the following form decomposition  for a type 1 subdiagonal algebra, which enables us to better understand the nest subalgebra component and the analytic component of the type 1 subdiagonal algebra.
\begin{definition}Let $\mathfrak A$ be a type 1 subdiagonal algebra. Then   $\mathfrak A$ is said to have the $(i,j)$-form if $F_-=iI$ and  $F_+=jI$  for $i,j=0,1$.
\end{definition}\label{d412}
\begin{theorem}\label{t413} Let $\mathfrak A$ be a type 1 subdiagonal algebra.  Then there exist unique pairwise orthogonal projections $\{C_{ij}:i,j=0,1\}\subseteq \mathcal Z(\mathcal M)\cap \mathfrak D$  with $\sum_{i,j=0}^1C_{ij}=I$ such that either $C_{ij}=0$ or
 $C_{ij}\mathfrak A\subseteq C_{ij}\mathcal M$  is a type 1 subdiagonal algebra with respect to $\Phi|_{C_{ij}\mathcal M}$ with the $(i,j)$-form.
\end{theorem}
\begin{proof} It is known that $F_{\pm}\in \mathcal Z(\mathcal M)\cap\mathfrak D$ by Corollary \ref{c410}.
 Put $C_{01}=F_+-F_-F_+$, $C_{10}=F_--F_-F_+$, $C_{11}=F_-F_+$ and $C_{00}=I-F_-\vee F_+=I-(C_{01}+C_{10}+C_{11})$. Then $C_{ij}\in \mathfrak D\cap\mathcal Z(\mathcal M)$ for $i,j=0,1$ by Lemme \ref{l46} and Corollary \ref{c410}.  Fixed  $i,j\in\{0,1\}$. if $C_{ij}\not=0$, then  $\mathfrak A_{ij}=C_{ij}\mathfrak A\subseteq C_{ij}\mathcal M$ is  a  type 1 subdiagonal algebra  with respect to $\Phi|_{C_{ij}\mathcal M}$. If we denote by $F_n^{ij}$  the projection defined as in (\ref{fn}) associated with $\mathfrak A_{ij}$, then $F^{ij}_n=C_{ij}\cap F_n$ and
$\lim\limits_{n\to\pm \infty}F_n^{ij}=C_{ij}F_{\pm}=F_{\pm}^{ij}$. Thus, $F_-^{01}=C_{01}F_-=0$, $F_+^{01}=C_{01}F_+=C_{01}$. This means that
$C_{01}\mathfrak A$ has $(0,1)$-form in $C_{01}\mathcal M$. Other cases are similar.

Assume  that there exist   pairwise orthogonal projections $\{E_{ij}:i,j=0,1\}\subseteq \mathcal Z(\mathcal M)\cap \mathfrak D$  with $\sum_{i,j=0}^1E_{ij}=I$ such that either $E_{ij}=0$ or
 $E_{ij}\mathfrak A\subseteq E_{ij}\mathcal M$  is a subdiagonal algebra with respect to $\Phi|_{E_{ij}\mathcal M}$ with $(i,j)$-form.  It is elementary that $(C\mathfrak A)_0=C\mathfrak A_0$ for any projection $C\in \mathcal Z(\mathcal M)\cap \mathfrak D$. Then $(C\mathfrak A)_0^n =C\mathfrak A_0^n$ and $CF_n$ is the projection defined as in (\ref{fn}) associated with $C\mathfrak A$. Thus $E_{01}F_-=\lim\limits_{n\to-\infty}E_{01}F_n=0$ and
 $E_{01}F_+=\lim\limits_{n\to+\infty}E_{01}F_n=E_{01}.$  Hence $E_{01}F_-=0$ and $E_{01}F_+=E_{01}$. It follows that $E_{01}\leq C_{01}$. By the same way, We easily have that $E_{ij}\leq C_{ij}$. Thus $E_{ij}=C_{ij}$ for any $i,j\in\{0,1\}$.
\end{proof}
  \begin{remark}\label{r414} (1) If $\mathfrak A$ has the $(1,1)$-form, then $F_n=I$ for all $n\in \mathbb Z$. For example, if $\mathfrak D$ is a factor, then $\mathfrak A$ has the $(1,1)$-form. In particular, $H^{\infty}(\mathbb T)$ has the $(1,1)$ form.

  (2)
  If $\mathfrak A$ has $(0,1)$-form, then $F_-=0$ and $F_+=I$. Since for any $n\geq1$, $F_n\geq F_+=I$, we have $\mathfrak M_n\not=\{0\}$. It follows  from Proposition \ref{p411} that there exists an infinite increasing  nest
   $\mathcal N=\{Q_n: 0\leq n\leq+\infty\}$ such that $\mathfrak A=\mbox{alg}_{\mathcal M}\mathcal N$. Similarly, if $\mathfrak A$ has $(1,0)$-form, then $F_-=I$ and $F_+=0$ and  there exists an infinite increasing  nest
   $\mathcal N=\{Q_n: -\infty \leq n\leq0\}$ such that $\mathfrak A=\mbox{alg}_{\mathcal M}\mathcal N$. For example, let $\mathcal H=\vee\{e_n:n\geq 0\}$, where $\{e_n:n\geq 0\}$ is an orthonormal  basis of $\mathcal H$. Denote by  $Q_{-n}$  the orthogonal projection from $\mathcal H$ onto $\vee\{e_k:k\geq n\}$ for any $n\geq0$ and $Q_{-\infty}=\{0\}$.  Then $\mbox{alg}\{Q_n:-\infty\leq n\leq0\}\subseteq \mathcal B(\mathcal H)$ has the $(1,0)$-form, while $\left(\mbox{alg}\{Q_n:-\infty\leq n\leq0\}\right)^*\subseteq \mathcal B(\mathcal H)$ has the $(0,1)$-form.

   (3) If $\mathfrak A$ has $(0,0)$-form, then $F_-=0$ and $F_+=0$.  In this case,  there exists a finite nest $\mathcal N$ such that
     $\mathfrak A=\mbox{alg}_{\mathcal M}\mathcal N$,  or we may choose both nests $\mathcal N_-=\{Q_n: -\infty\leq n\leq 0\}\subseteq \mathcal M$ and $\mathcal N_+=\{R_n:0\leq n\leq+\infty\}\subseteq \mathcal M$ such that $\mathfrak A=\mbox{alg}_{\mathcal M}\mathcal N_-=\mbox{alg}_{\mathcal M}\mathcal N_+$.    Let $M_n$ be the algebra of all complex $n\times n$ matrices  and $T_n$ the algebra of all upper triangular matrices for any $n\geq 2$.  Then $T_n\subseteq M_n$ is a subdiagonal algebra of  the $(0,0)$-form with a finite nest for all $n\geq2$.
      Put $\mathcal M=\oplus_{n\geq2}M_n$ and $\mathfrak A=\oplus_{n\geq 2}T_n$. It is elementary that $\mathfrak A\subseteq \mathcal M$ is a type 1 subdiagonal algebra of the $(0,0)$-form with an infinite nest.
       \end{remark}
  To consider a noncommutative  operator-theoretic variant of the Corona theorem for a type 1 subdiagonal algebra in general, we give  the following expectation motivated by \cite[Theorem 4.6]{solel2}, which may  be regarded as a noncommutative analogues of Arveson's expectation theorem(\cite[Proposition 5.1]{arv3}).
\begin{theorem}\label{t415}
Let $\mathfrak A$ be a type 1 subdiagonal algebra with  the $(1,1)$-form.  Then there is a conditional expectation $\Psi$ from $(L(\mathfrak D))^{\prime}$ onto $R(\mathcal M)$   such that

  (1) $\Psi((L(\mathfrak D))^{\prime}\cap \mbox{alg}\{P_n:n\in\mathbb Z\})=R(\mathfrak A)$;

   (2) $\Psi(P_n)=I$ for any $n\in \mathbb Z$;

   (3 )$\Psi(AP_n)=\Psi(A)$ for any $n\in\mathbb Z$ and
   $A\in (L(\mathfrak D))^{\prime}$.
 \end{theorem}
\begin{proof}
 (1)
Let $Z_+=\{n\in\mathbb Z: n\geq 0\}$ be the semigroup of nonnegative integers and let $\Lambda $ be a Banach limit on $\mathbb Z_+$. It is known that $\Lambda$ has the translation invariance, that is, $\forall g \in \ell^{\infty}(\mathbb Z_+)$, $\Lambda(g)=\Lambda (g_k)$, where $g_k(n)=g(k+n)$, $\forall k,n\in\mathbb Z_+$. This implies that $\Lambda $ also has  the right translation invariance, that is, for any $k< 0$, $\Lambda(g)=\Lambda (g_k)$,
   where $g_k(n)=\xi_n$ for some $\xi_n\in\mathbb C$ if $n<-k$, and  $g_k(n)=g(n+k)$ for all $n\geq -k$.

   Since $\mathfrak A$  has  the $(1,1)$-form,  $F_n=I$ and $\beta_n$ is a $*$-isomorphism for any $n\in\mathbb Z$.
For any $x,y\in L^2(\mathcal M)$, $A\in (L(\mathfrak D))^{\prime}$ and $n\in\mathbb Z_+$, put
  $g(A,x,y)(n)=\langle \beta_{-n}(A)x,y \rangle$.  It is trivial that  $\|g(A,x,y)\|\leq \|A\|\|x\|\|y\|$.  Then $g(A,x,y)\in \ell^{\infty}(\mathbb Z_+)$. We now
  define $[x,y]=\Lambda(g(A,x,y))$. It is trivial that $[\cdot,\cdot]$ is a bounded sesquilinear   form on $L^2(\mathcal M)$. Thus there is a bounded linear operator $\Psi(A)$ on $L^2(\mathcal M)$ such that $[x,y]=\langle \Psi(A)x,y\rangle$, $\forall x,y\in L^2(\mathcal M)$ and $\|\Psi(A)\|\leq \|A\|$.  For any unitary operator  $U\in L(\mathfrak D)$, $g(A,Ux,Uy)(n)=\langle \beta_{-n}(A)Ux,Uy\rangle=\langle \beta_{-n}(A)x,y\rangle=g(A,x,y)(n)$. This means that
  $$\langle U^*\Psi(A)Ux,y\rangle=\langle \Psi(A)Ux,Uy\rangle=\Lambda(g(A,Ux,Uy))=\Lambda(g(A,x,y))=\langle \Psi(A)x,y\rangle.$$ Thus $U^*\Psi(A)U=\Psi(A)$ and consequently,  $\Psi(A)\in (L(\mathfrak D))^{\prime}$.

  On the other hand, since $ \forall m\in \mathbb Z, A\in (L(\mathfrak D))^{\prime}$,
 $  \beta_m(A)=\sum_{k\geq1}U_{m,k}AU_{m,k}^*,$
    we have
   \begin{equation}\label{f38}
   U_{m,k}U_{m,k}^*\beta_m(A)=U_{m,k}AU_{m,k}^*, \ \ \forall m\in \mathbb Z, \forall A\in (L(\mathfrak D))^{\prime}.
   \end{equation}
    By formula (\ref{f38}) and Proposition \ref{p49}(3), we now have that
  \begin{align}\label{f39}
   &\ \ \ \ g(A, U_{m,k}^*x,U_{m,k}^*y)(n)=\langle \beta_{-n}(A)U_{m,k}^*x,U_{m,k}^*y\rangle\\
 \nonumber &=\langle U_{m,k}\beta_{-n}(A)U_{m,k}^*x,y\rangle=\langle U_{m,k}U_{m,k}^*\beta_m(\beta_{-n}(A))x,y\rangle\\
 \nonumber &=\langle L_{F_n}\beta_{m-n}(A)x,U_{m,k}U_{m,k}^*y\rangle=\langle \beta_{m-n}(A)x,U_{m,k}U_{m,k}^*y\rangle\\
 \nonumber &= \langle \beta_{-(-m+n)}(A)x,U_{m,k}U_{m,k}^*y\rangle=g_{-m}(A,x,U_{m,k}U_{m,k}^*y)(n),
  \end{align}
    where $g_{k}$ is a  translation of $g$ as above. Since $\Lambda $ is translation invariant, by formula (\ref{f39}),
  \begin{align*}\label{f310}
  &\ \ \ \ \langle U_{m,k}\Psi(A)U_{m,k}^*x,y\rangle=\langle \Psi(A)U_{m,k}^*x,U_{m,k}^*y\rangle\\
 \nonumber &=\Lambda(g(A,U_{m,k}^*x,U_{m,k}^*y))
  =\Lambda(  g_{-m}(A,x,U_{m,k}U_{m,k}^*y))\\
  \nonumber &=\Lambda(g(A,x, U_{m,k}U_{m,k}^*y))=\langle \Psi(A)x,U_{m,k}U_{m,k}^*y\rangle\\
  \nonumber &=\langle U_{m,k}U_{m,k}^*\Psi(A)x,y\rangle.
  \end{align*}
  It follows that
  $U_{m,k}\Psi(A)U_{m,k}^*=U_{m,k}U_{m,k}^*\Psi(A)$ for all $m\in\mathbb Z$ and $k\geq1$.  By summing both sides of this equation,
  we have
  \begin{equation}\label{f311}
  \beta_m(\Psi(A))=\sum_{k\geq 1}U_{m,k}\Psi(A)U_{m,k}^*=\sum_{k\geq1}U_{m,k}U_{m,k}^*\Psi(A)=F_m\Psi(A)=\Psi(A).\end{equation}
  Thus $\Psi(A)\in R(\mathcal M)$ by Proposition \ref{p49}(4). In particular, $\Psi\circ \Psi=\Psi$ from (\ref{f311}). Moreover, since $\beta_n(R_A)=F_nR_A=R_A$ for any $R_A\in R(\mathcal M)$ and $n\in\mathbb Z$,  $\Psi$ is a positive projection from $(L(\mathfrak D))^{\prime} $ onto $R(\mathcal M)$.  Hence is a conditional expectation by  \cite[Theorem 1]{to}.

  Take any $A\in  (L(\mathfrak D))^{\prime}\cap \mbox{alg}\{P_n: n\in \mathbb Z\}$ and $n,m\in\mathbb Z$. Since $\beta_n(P_m)=L_{F_n}P_{n+m}=P_{n+m}$ by Proposition \ref{p49}(5),
 \begin{align*}
 &\ \ \ \ \beta_n(A)P_m=\beta_n(A)\beta_n(P_{m-n})=\beta_n(AP_{m-n})=\beta_n(P_{m-n}AP_{m-n})\\
 &=\beta_n(P_{n-m})\beta_n(A)\beta_n(P_{n-m})=P_m\beta_n(A)P_m,
 \end{align*} which  implies that
 $g(A,P_mx,y)(n)=g(A, P_mx,P_my)(n)$. Thus $\langle\Psi(A)P_mx,y\rangle=\langle \Psi(A)P_mx,P_my\rangle$ and hence $\Psi(A)P_m=P_m\Psi(A)P_m$.  In particular, $\Psi(A)P_0=P_0\Psi(A)P_0$.  By \cite[Theorem 2.2]{js}, $\Psi(A)\in R(\mathfrak A)$.

 (2) Since $\beta_{-n}(P_m)=P_{m-n}$, $n,m\in\mathbb Z$ by Proposition \ref{p49}(5) again, $g(P_m,x,y)(n)=\langle \beta_{-n}(P_m)x,y\rangle=\langle P_{m-n}x,y\rangle$ and $\lim\limits_{n\to+\infty}\langle P_{m-n}x,y\rangle=\langle x,y\rangle$. Thus $\Lambda(g(P_m,x,y)=\langle x,y\rangle$. Hence
 $\langle \Psi(P_m)x,y\rangle=\Lambda(g(P_m,x,y))=\langle x,y\rangle$. This implies that  $\Psi(P_m)=I$.

 (3) This is a direct consequence of Schwarz inequality. In fact,
 $$(\Psi(A(I-P_n)))^*\Psi(A(I-P_n)\leq \Psi((I-P_n)A^*A(I-P_n))\leq\|A\|^2\Psi(I-P_n)=0.$$ Thus $\Psi(AP_n)=\Psi(A)$, $\forall A\in (L(\mathfrak D))^{\prime}$, $\forall n\in \mathbb Z$.
\end{proof}
\begin{remark}\label{r416}  We also can obtain
  a conditional expectation $\Gamma:(R(\mathfrak D))^{\prime}\to L(\mathcal M)$    with  similar properties as in Theorem \ref{t415} by use of properties of $\{\eta_n:n\in\mathbb Z\}$.
 \end{remark}

\section{A noncommutative    operator-theoretic  variant of  the Corona theorem}
We consider a noncommutative    operator-theoretic  variant of  the Corona theorem for type 1 subdiagonal algebras in general.
We firstly give a noncommutative analogues of \cite[Proposition 5.2]{arv3}.
Let $P_0(L(\mathfrak D))^{\prime}P_0$(resp. $P_0(R(\mathfrak D))^{\prime}P_0$) be the von Neumann algebra of the restriction of $(L(\mathfrak D))^{\prime}$(resp. $R(\mathfrak D)$) on $H^2$.
 $\mathfrak L=\{T_A: A\in \mathfrak A\}$ and $\mathfrak R=\{t_A:A\in\mathfrak A\}$ denote the left and right analytic Toeplitz operator algebras respectively. It is elementary that  $t_A\in P_0(L(\mathfrak D))^{\prime}P_0$(resp. $T_A\in P_0(R(\mathfrak D))^{\prime}P_0$).  Since $P_0P_n=P_n$ for all $n\geq1$,  we still denote by $P_n$ the restriction of $P_n$ on $H^2$ for all $n\geq1$.
 \begin{proposition}\label{p51}Let $\mathfrak A$ be a type 1 subdiagonal algebra  with the  $(1,1)$-form. Then there exists a positive linear projection  $\psi$  from $P_0(L(\mathfrak D))^{\prime}P_0|_{H^2}$ onto the space  $\{t_A: A\in \mathcal M\}$  of all  right Toeplitz  operators  such that

 (1) $\psi(I)=I,$ $\|\psi\|=1$,

 (2) $\psi \left(P_0(L(\mathfrak D))^{\prime}P_0\cap\mbox{alg}\{P_n: n\geq1\}\right)\subseteq \mathfrak R$,

  (3)
    $\psi(t_BA)=t_B\psi(A)$, $\forall A\in P_0(L(\mathfrak D))^{\prime}P_0\cap\mbox{alg}\{P_n: n\geq1\}$, $\forall B\in\mathfrak A$.
 \end{proposition}
 \begin{proof} (1)
 Note that $P_0$ is the projection from $L^2(\mathcal M)$ onto $H^2$ and $P_0\in (L(\mathfrak D))^{\prime}$. Then for any
 $A\in P_0(L(\mathfrak D))^{\prime}P_0$,   $\Psi(P_0AP_0)\in R(\mathcal M)$. We may define $\psi(P_0AP_0)=P_0\Psi(P_0AP_0)|_{H^2}\in\{t_B: B\in\mathcal M\}$.
 This means that $\psi $ is a positive linear map of norm 1 with $\psi(I)=I$ by Theorem \ref{t415}.
 Moreover, for any $B\in\mathcal M$, $t_BP_0=P_0R_BP_0\in P_0(L(\mathfrak D))^{\prime}P_0.$
 Then
  $$\psi(P_0R_BP_0)=P_0\Psi(P_0R_BP_0)|_{H^2}=P_0\Psi(R_BP_0)|_{H^2}=P_0\Psi(R_B)|_{H^2}=P_0R_B|_{H^2}=t_B$$ by Theorem \ref{t415} again.

  (2) If $A\in P_0(L(\mathfrak D))^{\prime}P_0\cap \mbox{alg}\{P_n: n\geq1\}$, then $A=P_0AP_0\in(L(\mathfrak D))^{\prime}$  and hence $A\in (L(\mathfrak D))^{\prime}\cap \mbox{alg}\{P_n: n\in \mathbb Z\}$. Thus $\Psi(A)=R_C\in R(\mathfrak A)$ for some $C\in\mathfrak A$  by Theorem \ref{t415}. It follows that  $\psi(A)=P_0\Psi(A)|_{H^2}=t_{C}\in\mathfrak R$.

 (3)  If    $A\in P_0(L(\mathfrak D))^{\prime}P_0\cap \mbox{alg}\{P_n: n\geq1\}$, then   $\Psi(A)=R_C$ for some $C\in \mathfrak A$ by (2).  In particular, $\Psi(A)P_0=P_0\Psi(A)P_0$.
    Now for any  $B\in\mathfrak A$,
  \begin{align*}
  \psi(t_BA)&=P_0\Psi(P_0R_BP_0AP_0)|_{H^2}=P_0\Psi(R_BP_0AP_0)|_{H^2}\\
  &=P_0R_B\Psi(A)|_{H^2}=(P_0R_B\Psi(A)P_0)|_{H^2}\\
  &=P_0R_BP_0\Psi(A)P_0|_{H^2}
  =t_B\psi(A).
  \end{align*}
 \end{proof}
 \begin{proposition}\label{p52} Let $\mathfrak A$ be a type  1  subdiagonal algebra  with $(1,1)$-form and $\{A_k:1\leq k\leq N\}\subseteq \mathfrak A$. If there exists a positive number $\varepsilon$ such that
\begin{equation}\label{f51}
\sum\limits_{k=1}^N\|t_{A_k}^*x\|^2\geq \varepsilon^2 \|x\|^2, \forall x\in H^2,
\end{equation}  then there exist operators $\{B_k:1\leq k\leq N\}\subseteq\mathfrak A$
 such that
 \begin{equation*}
 \sum\limits_{k=1}^NB_kA_k=I
 \end{equation*} with $\|B_k\|\leq 4N\alpha \varepsilon^{-3}$.
 \end{proposition}
 \begin{proof}
 Let $\mathcal N$  be the von Neumann algebra   $P_0L(\mathfrak D))^{\prime}P_0$ on $H^2$. Since $P_n\in\mathcal N$ for all $n\geq1$, $Q_n=I-P_n\in\mathcal N$ and $\{Q_n: 0\leq n\leq+\infty\}$ is a nest in $\mathcal N$, where $Q_0=0$ and $Q_{+\infty}=I$.
 Thus $\mbox{alg}_{\mathcal N}\{Q_n: n\geq0\}$ is a nest subalgebra in $\mathcal N$ with $t_A^*=t_{A^*}\in \mathcal N\cap\mbox{alg}\{Q_n: n\geq0\}$ for any $A\in\mathfrak A$. We next claim that operators $\{t_{A_k^*}:1\leq k\leq N\}$ satisfy  conditions of Theorem \ref{t2}. Note that $I-Q_n=P_n$ for $n\geq1$ and $I-Q_0=I$.  Then $\sum\limits_{k=1}^N\|(I-Q_0)t_{A_k^*}x\|^2 =\sum\limits_{k=1}^N\|t_{A_k^*}x\|^2\geq \varepsilon^2 \|x\|^2$, $\forall x\in H^2$. Assume that $n\geq1$.  Note that for any partial isometry $V\in\mathfrak A$ such that $V^*V\in\mathfrak D$, $T_V$ is a partial isometry such that $T_V^*T_V=T_{V^*V}$. Now
 $$P_n(H^2)=\oplus_{k\geq1}^{col}U_{n,k}H^2=\oplus_{k\geq1}^{col}T_{U_{n,k}}H^2.$$
 Then $$P_nH^2=\sum_{k\geq 1}T_{U_{n,k}}T_{U_{n,k}}^*T_{U_{n,k}}H^2=\left(\sum_{k\geq 1}T_{U_{n,k}}T_{U_{n,k}}^*\right) H^2.$$
  Thus $P_n=\sum_{k\geq 1}T_{U_{n,k}}T_{U_{n,k}}^*$. On the other hand, $T_At_B=t_BT_A$ for all $A,B\in \mathfrak A$ by \cite[Theorem 2.3]{ji2}. We then have
 \begin{align}\label{f52}
 \ \ \ \ &\|(I-Q_n)t_A^*x\|^2=\|P_nt_A^*x\|^2=\|(\sum_{k\geq1}T_{U_{n,k}}T_{U_{n,k}}^*)t_A^*x\|^2\\
 \nonumber &=\sum_{k\geq1}\|T_{U_{n,k}}^*t_A^*x\|^2
 =\sum_{k\geq1}\|t_A^*T_{U_{n,k}}^*x\|^2.
  \end{align}
  It follows from  formulae (\ref{f51})  and (\ref{f52}) that
  \begin{align*}
  &\ \ \ \ \sum_{j=1}^N\|(I-Q_n)t_{A_j}^*x\|^2=\sum_{j=1}^N\sum_{k\geq1}\|t_{A_j}^*T_{U_{n,k}}^*x\|^2\\
  &=\sum_{k\geq1}\sum_{j=1}^N\|t_{A_j}^*T_{U_{n,k}}^*x\|^2\\
  &\geq \varepsilon^2 \sum_{k\geq1} \|T_{U_{n,k}}^*x\|^2\\
  &=\varepsilon^2 \sum_{k\geq1} \|T_{U_{n,k}}T_{U_{n,k}}^*x\|^2\\
  &=\varepsilon^2 \|(\sum_{k\geq1}T_{U_{n,k}}T_{U_{n,k}}^*)x\|^2\\
  &=\varepsilon^2\|P_nx\|^2=\varepsilon^2\|(I-Q_n)x\|^2.
  \end{align*}
  By Theorem \ref{t2}, there exist operators $\{C_k:1\leq k\leq N\}\subseteq  \mbox{alg}_{\mathcal N}\{Q_n: n\geq0\}$   such that
 $\sum_{k\geq1}C_kt_{A_k}^*=I$. Hence $\sum_{k\geq1}t_{A_k}C_k^*=I$. Note that $\{C_k^*:1\leq k\leq N\}\subseteq  \mbox{alg}_{\mathcal N}\{P_n: n\geq0\}$. By Proposition \ref{p51}, $\psi(\sum_{k\geq1}t_{A_k}C_k^*)=\sum_{k\geq1}t_{A_k}\psi(C_k^*)=I$. By Proposition \ref{p51} again,
 $\psi(C_k^*)=t_{B_k}\in \mathfrak R$ For some $B_k\in \mathfrak A$  for any $1\leq k\leq N$. Therefore,
 $$\sum_{k\geq1}t_{A_k}\psi(C_k^*)=\sum_{k\geq1}t_{A_k}t_{B_k}=t_{\sum_{k\geq1}B_kA_k}=I.$$ Consequently,
 $\sum_{k\geq1}B_kA_k=I$.  Since $\|T_A\|=\|t_A\|=\|A\|$  for any $A\in\mathcal M$ by \cite[Theorem 4.3]{pru},  we easily have $\|B_k\|\leq 4N\alpha \epsilon^{-3}$ for any $1\leq k\leq N$.
 \end{proof}
 Combine Remark \ref{r35},  Theorem \ref{t413} with  Proposition  \ref{p52}, we summarize the following noncommutative operator-theoretic variant of the Corona theorem for a type 1 subdiagonal algebra.
\begin{theorem}\label{t52} Let $\mathfrak A$ be a type  1  subdiagonal algebra     and $\{A_k:1\leq k\leq N\}\subseteq \mathfrak A$. If there exists a positive number $\varepsilon$ such that
\begin{equation}\label{f53}
\sum\limits_{k=1}^N\|t_{A_k}^*x\|^2\geq \varepsilon^2 \|x\|^2, \forall x\in H^2,
\end{equation}  then there exist operators $\{B_k:1\leq k\leq N\}\subseteq\mathfrak A$
 such that
 \begin{equation*}
 \sum\limits_{k=1}^NB_kA_k=I
 \end{equation*}  with $\|B_k\|\leq 4N\alpha \varepsilon^{-3}$.
 \end{theorem}
 \begin{proof}
Let $C_{ij}\in\mathcal Z(\mathcal M)\cap \mathfrak D$ be projections in Theorem \ref{t413}.  Then we have $H^2=\oplus_{i,j=0}^1C_{ij}H^2$ and
$t_A=\oplus_{i,j=0}^1C_{ij}t_A$ for any $A\in\mathcal M$. In particular,
\begin{equation*}
\sum\limits_{k=1}^N\|C_{ij}t_{A_k}^*x\|^2\geq \varepsilon^2 \|x\|^2, \forall x\in C_{ij}H^2.
\end{equation*}
Since $C_{ij}\mathfrak A$ is a nest subalgebra of $C_{ij}\mathcal M$ with a nest  which is order isomorphic to a sublattice of $\mathbb Z\cup\{-\infty,+\infty\}$  when either $i=0$ or  $j=0$ by Remark \ref{r414}, it follows  from Remark \ref{r35} that  there exist
$\{B_k^{ij}:1\leq k\leq N\}\subseteq C_{ij}\mathfrak A$ such that $\sum_{k=1}^NB_k^{ij}C_{ij}A_k=C_{ij}$. Now $C_{11}\mathfrak A$ is a type 1 subdiagonal algebra with  the $(1,1)$-form. Thus by Proposition \ref{p52}, $\sum_{k=1}^NB_k^{11}C_{11}A_k=C_{11}$ for some $\{B_k^{11}:1\leq k\leq N\}\subseteq C_{11}\mathfrak A$. Let $B_k=\sum_{i,j=0}^1B_k^{ij}$ for any  $1\leq k\leq N$. We have $\{B_k:1\leq k\leq N\}\subseteq \mathfrak A$
such that $\sum_{k=1}^NB_kA_k=I.$  The  inequality   $\|B_k\|\leq 4N\alpha \varepsilon^{-3}$ can be  easily derived.
 \end{proof}
 \begin{remark}\label{r5} If we consider an expectation from $(R(\mathfrak D))^{\prime}$ onto $L(\mathcal M)$ as in Remark \ref{r416}, then there exists a positive linear projection   $\gamma$ from $P_0(R(\mathfrak D))^{\prime}P_0$ onto the space  $\{T_A: A\in \mathcal M\}$  of all  left  Toeplitz  operators  with those properties as in Proposition \ref{p51}.
Therefore, if we replace right Toeplitz operators $\{t_{A_k^*}: 1\leq k\leq n\}$  by left Toeplitz operators $\{T_{A_k^*}: 1\leq k\leq n\}$ in (\ref{f53}),    then there exist operators $\{B_k:1\leq k\leq N\}\subseteq\mathfrak A$
 such that
 $ \sum\limits_{k=1}^NA_kB_k=I.$

 \end{remark}

\bibliographystyle{amsplain}

\end{document}